\documentclass{elsarticle}
\usepackage{graphicx,graphics}
\usepackage{pifont}
\usepackage{amsthm}
\usepackage{algorithm}
\usepackage{algorithmic}
\DeclareFontFamily{OML}{script}{}
\DeclareFontShape{OML}{script}{m}{it}
{ <5-20> rsfs10 }{}
\DeclareMathAlphabet{\mathcalipt}{OML}{script}{m}{it}

\renewcommand{\mathcal}[1]{{\mathcalipt #1}\hspace{0.2ex}}
\usepackage{color}
\ifx\red\undefined
\newcommand{\red}{\color{red}\huge}

\usepackage{subcaption}
\usepackage{float}
\usepackage{amscd}
\usepackage{amsmath}
\usepackage{latexsym}
\usepackage{amsfonts}
\usepackage{amssymb}
\usepackage{color}
\usepackage{multicol}
\usepackage{bm}
\allowdisplaybreaks[4]
\newcommand{\dif}{\mathrm{d}}
\newcommand{\p}{\partial}
\newcommand{\xx}{{\bm x}}
\newcommand{\y}{{\bm y}}
\usepackage[utf8]{inputenc}
\newcommand{\re}[1]{\mbox{\rm$($\ref{#1}$)$}}

\newcommand{\Rmnum}[1]{\uppercase\expandafter{\romannumeral #1}}

\newtheorem{thm}{Theorem}[section]

\newtheorem{lem}{Lemma}[section]

\newtheorem{rem}{Remark}[section]

\def\theequation{\arabic{section}.\arabic{equation}}
\numberwithin{equation}{section}

\begin{document}
\title{Convergence analysis of neural networks for solving a free boundary problem}
\author{Xinyue Evelyn Zhao}
\address{Department of Applied and Computational Mathematics and Statistics,\\
University of Notre Dame, Notre Dame, IN 46556, USA\\
xzhao6@nd.edu}

\author{Wenrui Hao}\address{Department of Mathematics, 
The Pennsylvania State University, \\
University Park, PA 16802, USA\\
wxh64@psu.edu}

\author{Bei Hu}
\address{Department of Applied and Computational Mathematics and Statistics,\\
University of Notre Dame, Notre Dame, IN 46556, USA\\
b1hu@nd.edu}

\begin{abstract}
Free boundary problems deal with systems of partial differential equations,
where the domain boundaries are apriori unknown. Due to this special characteristic, it is challenging to solve the free boundary problems either theoretically or numerically.  In this paper, we develop a novel approach for solving a modified Hele-Shaw problem based on the neural network discretization. The existence of the numerical solution with this discretization is established theoretically. We also numerically verify this approach by computing the symmetry-breaking solutions which are guided by the bifurcation analysis near the radially-symmetric branch. Moreover, we further verify the capability of this approach  by computing some non-radially symmetric solutions which are not characterized by any theorems. 
\end{abstract}
\maketitle

\section{Introduction}

Many mathematical models of natural phenomena, e.g., biology, physics and materials science, involve the solutions of systems of
partial differential equations (PDEs) with free (moving) boundaries \cite{F1,FH,HCF,HF,HHHS}. Among these free boundary problems, the generalized Hele-Shaw problem with surface tension is the most popular and widely-studied problem with various applications ranging from physics to biology \cite{F1,friedman2012variational}. This problem has attracted extensive experimental
and mathematical studies since the initial work of Saffman and
Taylor in 1958 \cite{ST}, based on the experimental innovation of
confining a fluid between two closely-spaced plates by Hele-Shaw
\cite{HSH}.
 From a mathematical point of view, studies
of this problem can be formulated both numerically and theoretically to focus on the solutions and their structures \cite{chen2003free,dibenedetto1984ill}. In the
last few decades,  generalized Hele-Shaw problems with surface tension have been formulated from biological and physical modeling \cite{constantin1993global,F1}. Thus theories and nonlinear simulations of these problems have been developed to understand the structure of steady-state solutions. Although the PDE theory can help in some
special cases, the in-depth study of these problems often requires large-scale simulations including numerically computing steady-state solutions \cite{HHHMS,HHHS}. Efficient numerical methods for computing the steady-state solutions \cite{wang2018two}, bifurcations \cite{HHHLSZ}, and stability are keys to
understanding these systems. Underlying all these is the common grand challenge of developing efficient numerical algorithms for complex PDE systems with free (moving) boundaries.

Recently, there are several numerical methods developed for studying the generalized Hele-Shaw problem with surface tension, e.g.,  computing
multiple steady-states by coupling multi-grid and domain decomposition techniques with numerical algebraic geometry \cite{HHHMS,HHHS,HHS1,wang2018two}; detecting bifurcation points by using the adaptive homotopy tracking method \cite{HHHLSZ,hao2017homotopy,HZ}; and exploring their global solution
structures based on PDE theories \cite{HHHLSZ,HHHS,HHLS}. These numerical methods have also been successfully applied
to some complex biological networks including tumor growth model and cardiovascular disease risk evaluation \cite{HCF,HF}. We also analyzed the boundary integral method on a simplified Hele-Shaw problem without the surface tension term \cite{HHLS} and provided a rigorous convergence analysis.
However, to the date, there are still several numerical challenges for solving the generalized Hele-Shaw problem with surface tension:
\begin{itemize}
  \item[1)] there is lack of rigorous theoretical analysis of
numerical methods for free boundary problems with the surface tension;
  \item[2)] and steady-state solution patterns are hard to compute so that
the global solution structure is unclear.
\end{itemize}  Therefore,
efficient numerical methods, rigorous theoretical analysis of these
numerical methods and global solution structures are
 needed to deeply study the generalized Hele-Shaw problem.

Machine learning has been experiencing an extraordinary resurgence in many
important artificial intelligence applications since the
late 2000s. In particular, it has been able to produce
state-of-the-art accuracy in computer vision \cite{sebe2005machine}, video
analysis \cite{camastra2015machine}, natural language processing \cite{indurkhya2010handbook} and
speech recognition \cite{amodei2016deep}. Recently, interest in machine learning based approaches in the applied mathematics community has increased rapidly \cite{bright2013compressive,raissi2017machine}. This
growing enthusiasm for machine learning  stems from  massive amounts of data available from scientific computations \cite{dillenbourg1999collaborative} and other sources \cite{durrleman2014morphometry}; the design of efficient data
analysis algorithms \cite{yang2015oscillatory}; advances in high-performance computing; and the data-driven modeling \cite{marsland2015machine}. In order to take advantage of machine learning, we will develop a new approach for solving free boundary problems and address the current challenges in this area.

\section{The free boundary problem and bifurcation analysis}
\subsection{The model problem}
The classical Hele-Shaw problem seeks a fluid domain $\Omega(t)\in \mathbb{R}^2$ and the fluid pressure $\sigma$ such that
\begin{equation}\label{mwa}
\left\{
    \begin{array}{ll}
    \Delta \sigma = 0 \hspace{2em}    &\text{in }\Omega(t),\\
    \sigma = \kappa \hspace{2em}  &\text{on } \p \Omega(t),\\
    V_n = -\frac{\p \sigma}{\p {\bm n}} \hspace{2em}    &\text{on }\p \Omega(t),
    \end{array}
    \right.
\end{equation}
where $\kappa$ denotes the curvature of $\p \Omega(t)$ ($\kappa = \frac1R$ if $\p \Omega(t)$ is a circle of radius $R>0$); and $V_n$ is the velocity of the fluid boundary $\p \Omega(t)$ in the outward normal direction ${\bf n}$.

It is well-known that model \re{mwa} possesses only radially symmetric stationary solution. In order to investigate the complexity of free boundaries, we introduce a modified Hele-Shaw model as below:
\begin{equation}\label{model1}
\left\{
    \begin{array}{ll}
    -\Delta \sigma = c(-\sigma - \mu) \hspace{2em} &\text{in } \Omega(t),\\
    \sigma = \kappa \hspace{2em} &\text{on } \p \Omega(t),\\
    V_n = -\frac{\p \sigma}{\p {\bm n}} + \beta \hspace{2em} &\text{on } \p \Omega(t),
    \end{array}
    \right.
\end{equation}
where $c, \mu,\beta >0$. The first equation on the right-hand side of
\re{model1} represents a sink of fluid, while the additional constant $\beta$ represents the influx of fluid in addition to the balance of mass. When $c=\beta = 0$, \re{model1} reverts to the classical Hele-Shaw problem. By introducing the non-dimensional length scale $L_D = \sqrt{c}$, we define:
$$\tilde{\xx} = L_D \xx,\hspace{2em} \tilde{\sigma}(\tilde{\xx}) = \sigma(\xx) + \mu, \hspace{2em} \tilde{\Omega}(t) = L_D \Omega(t), \hspace{2em} \tilde{\beta} = \frac{\beta}{L_D}.$$
After dropping the $\sim$ in the above variables, the non-dimensional model takes the following form
\begin{equation}\label{model12}
\left\{
    \begin{array}{ll}
    -\Delta \sigma = -\sigma \hspace{2em} &\text{in } \Omega(t),\\
    \sigma = \mu + \kappa \hspace{2em} &\text{on } \p \Omega(t),\\
    V_n = -\frac{\p \sigma}{\p {\bm n}} + \beta \hspace{2em} &\text{on } \p \Omega(t).
    \end{array}
    \right.
\end{equation}

We consider the steady state system of \re{model12} by setting $V_n=0$ and obtain the following stationary system:
\begin{equation}\label{model2}
\left\{
    \begin{array}{ll}
    -\Delta \sigma = -\sigma \hspace{2em} &\text{in } \Omega,\\
    \sigma = \mu + \kappa \hspace{2em} &\text{on } \p \Omega,\\
    \frac{\p \sigma}{\p {\bm n}} = \beta \hspace{2em} &\text{on } \p \Omega.
    \end{array}
    \right.
\end{equation}
Theoretically, system \re{model2} admits a unique radially symmetric solution $\sigma_S(r)$ with radius $r=R_S$:
\begin{equation}\label{sigmaS}
    \sigma_S(r) = \Big(\mu+\frac1{R_S}\Big) \frac{I_0(r)}{I_0(R_S)},
\end{equation}
provided that $\beta = \beta(\mu,R_S)$ is given by
\begin{equation}\label{beta}
\beta = \Big(\mu+\frac1{R_S}\Big) \frac{I_1(R_S)}{I_0(R_S)}.
\end{equation}
Here $I_n(r)$ is the modified Bessel function for integer $n\ge 0$.

We are more interested in finding the non-radially symmetric solutions of system \re{model2}. Particularly, we would like to know what the boundaries look like in non-radially symmetric case. In this section, we shall carry out a theoretical bifurcation analysis by using the Crandall Rabinowitz Theorem (Theorem \ref{bifurthm} in Appendix) to show there exists branches of symmetry-breaking solutions to system \re{model2}; and in the next section, we will propose a new method, which is a combination of boundary integral method (BIM) and machine learning approximation, to numerically derive the shapes of the boundaries of system \re{model2}.

\subsection{Bifurcation results}
To begin with, we consider a family of domains with perturbed boundaries in polar coordinates
\begin{equation*}
    \p \Omega_\varepsilon\,:\, r=R_S + \tilde{R}(\theta) = R_S + \varepsilon S(\theta).
\end{equation*}
Let $\sigma$ be the solution of the system
\begin{equation}\label{model3}
\left\{
    \begin{array}{ll}
    -\Delta \sigma = -\sigma \hspace{2em} &\text{in } \Omega_\varepsilon,\\
    \sigma = \mu + \kappa \hspace{2em} &\text{on } \p \Omega_\varepsilon,
    \end{array}
    \right.
\end{equation}
and define $\mathcal{F}$ by
\begin{equation}
    \label{mathcalF}
    \mathcal{F}(\tilde{R},\mu) = \frac{\p \sigma}{\p {\bm n}}\Big|_{\p \Omega_\varepsilon} - \beta.
\end{equation}
Based on system \re{model2}, 
$\sigma$ is a symmetry-breaking solution of system \re{model2} if and only if $\mathcal{F}(\tilde{R},\mu) = 0$. 

Following \cite{angio,Fengjie,FH3,Hongjing,Zejia,zhao2}, it can be established that $\sigma$ admits the following expansion
\begin{equation}\label{expand}
    \sigma = \sigma_S + \varepsilon \sigma_1 + O(\varepsilon^2),
\end{equation}
where $\sigma_1$ is the solution to the linearized system ($B_{R_S}$ denotes the disk centered at 0 with radius $R_S$)
\begin{equation}\label{linear}
\left\{
    \begin{array}{ll}
    -\Delta \sigma_1 = -\sigma_1 \hspace{2em} &\text{in } B_{R_S},\\
    \sigma_1 = -\frac{1}{R_S^2}(S+S_{\theta\theta}) - \frac{\p \sigma_S(R_S)}{\p r} S \hspace{2em} &\text{on } \p B_{R_S}.
    \end{array}
    \right.
\end{equation}
We then substitute \re{expand} into \re{mathcalF} to obtain
\begin{eqnarray*}
    \mathcal{F}(\tilde{R},\mu) &=& \frac{\p \sigma_S(R_S)}{\p r} +  \frac{\p^2 \sigma_S(R_S)}{\p r^2}\varepsilon S(\theta) +  \frac{\p \sigma_1(R_S,\theta)}{\p r}\varepsilon -\beta + O(\varepsilon^2)\\
    &=& \mathcal{F}(0,\mu) + \varepsilon \Big(\frac{\p^2 \sigma_S(R_S)}{\p r^2}S(\theta) + \frac{\p \sigma_1(R_S,\theta)}{\p r}\Big) + O(\varepsilon^2),
\end{eqnarray*}
which formally gives the Fr\'echet derivative of $\mathcal{F}$ as: 
\begin{equation}
    \label{FreD}
    [\mathcal{F}_{\tilde{R}}(0,\mu)]S(\theta) = \frac{\p^2 \sigma_S(R_S)}{\p r^2}S(\theta) + \frac{\p \sigma_1(R_S,\theta)}{\p r}.
\end{equation}
In what follows, we shall use \re{FreD} to establish the bifurcation points by verifying the regularity and the four assumptions in the Crandall-Rabinowitz Theorem.

Like in \cite{zhao2}, we introduce the Banach spaces:
\begin{gather*}
    X^{l+\alpha} = \{S\in C^{l+\alpha}(B_1), S \text{ is $2\pi$-periodic in $\theta$}\},\nonumber\\
    \label{Banach}
    X^{l+\alpha}_2 = \text{closure of the linear space spanned by $\{\cos(n\theta),n=0,2,4,\cdots\}$ in $X^{l+\alpha}$},
\end{gather*}
and set the perturbation $S(\theta) = \cos(n\theta)$. Using a separation of variables, we seek a solution of the form
\begin{equation}\label{s1}
    \sigma_1(r,\theta) = \sigma_1^n(r)\cos(n\theta).
\end{equation}
Based on \re{linear}, $\sigma_1^n$ satisfies
\begin{equation}\label{sigma1n}
\left\{
    \begin{array}{ll}
    -\frac{\p^2 \sigma_1^n}{\p r^2} - \frac1r \frac{\p \sigma_1^n}{\p r} + \frac{n^2}{r^2}\sigma_1^n = -\sigma_1^n \hspace{2em} &\text{in } B_{R_S},\\
    \sigma_1^n = \frac{n^2-1}{R_S^2} - \frac{\p \sigma_S(R_S)}{\p r} \hspace{2em} &\text{on } \p B_{R_S},
    \end{array}
    \right.
\end{equation}
and is explicitly given by
\begin{equation}
    \label{s1n}
    \sigma_1^n(r) = \Big[\frac{n^2-1}{R_S^2} - \frac{\p \sigma_S(R_S)}{\p r}\Big]\frac{I_n(r)}{I_n(R_S)} = \Big[ \frac{n^2-1}{R_S^2} - \Big(\mu + \frac1{R_S}\Big)\frac{I_1(R_S)}{I_0(R_S)}\Big]\frac{I_n(r)}{I_n(R_S)}.
\end{equation}
Substituting \re{sigmaS}, \re{s1}, and \re{s1n} into \re{FreD}, we obtain
\begin{equation}
\label{FreD1}
\begin{split}
    [\mathcal{F}_{\tilde{R}}(0,\mu)]\cos(n\theta) &\,=\, \bigg[\Big(\mu + \frac1{R_S}\Big) \frac{I_1'(R_S)}{I_0(R_S)} + \Big[ \frac{n^2-1}{R_S^2} -\Big(\mu + \frac1{R_S}\Big)\frac{I_1(R_S)}{I_0(R_S)}\Big]\frac{I_n'(R_S)}{I_n(R_S)}    \bigg]\cos(n\theta)\\
    &\,=\, \bigg[-\mu \frac{I_1(R_S)}{I_0(R_S)}\Big(\frac{I_n'(R_S)}{I_n(R_S)}-\frac{I_1'(R_S)}{I_1(R_S)}\Big) \\
    &\hspace{5.2em}+ \frac{n^2-1}{R_S^2}\frac{I_n'(R_S)}{I_n(R_S)} - \frac{I_1(R_S)}{R_S I_0(R_S)}\Big(\frac{I_n'(R_S)}{I_n(R_S)}-\frac{I_1'(R_S)}{I_1(R_S)}\Big) \bigg]\cos(n\theta).
    \end{split}
\end{equation}
It follows that for $n\neq 1$, $[\mathcal{F}_{\tilde{R}}(0,\mu)]\cos(n\theta)=0$ if and only if 
\begin{equation}\label{bifurps}
    \begin{split}
    \mu &= \mu_n(R_S) \\ 
    &\,\triangleq\, \Big[ \frac{n^2-1}{R_S^2}\frac{I_n'(R_S)}{I_n(R_S)} - \frac{I_1(R_S)}{R_S I_0(R_S)}\Big(\frac{I_n'(R_S)}{I_n(R_S)}-\frac{I_1'(R_S)}{I_1(R_S)}\Big)\Big]\Big/ \Big[ \frac{I_1(R_S)}{I_0(R_S)}\Big(\frac{I_n'(R_S)}{I_n(R_S)}-\frac{I_1'(R_S)}{I_1(R_S)}\Big)\Big]\\
    &\,=\, -\frac{1}{R_S} + \frac{I_0(R_S)}{R_S^2 I_1(R_S)} \cdot
    \Big[(n^2-1)\frac{I_n'(R_S)}{I_n(R_S)}\Big]\Big/\Big[\frac{I_n'(R_S)}{I_n(R_S)}-\frac{I_1'(R_S)}{I_1(R_S)}\Big]\\
    &\,=\, -\frac{1}{R_S} + \frac{I_0(R_S)}{R_S^2 I_1(R_S)}\cdot \frac{I_n'(R_S)}{I_n(R_S)}  \Big/\Big[ \frac{1}{n^2-1}\Big(\frac{I_n'(R_S)}{I_n(R_S)}-\frac{I_1'(R_S)}{I_1(R_S)} \Big)\Big].
    \end{split}  
\end{equation}
In order to analyze $\mu_n$, we recall two inequalities from \cite{borisovich}, namely,
\renewcommand{\theequation}{[7,(A.1)]}
\begin{equation}
\frac{I_{n+1}'(r)}{I_{n+1}(r)} > \frac{I_n'(r)}{I_n(r)}\hspace{2em} \text{for all $n\ge0$ and $r>0$},
\end{equation}
\addtocounter{equation}{-1}
\renewcommand{\theequation}{\thesection.\arabic{equation}}
\renewcommand{\theequation}{[7,(A.7)]}
\begin{equation}
\frac{1}{n^2-1}\Big(\frac{I_n'(r)}{I_n(r)}-\frac{I_1'(r)}{I_1(r)}\Big) > \frac{1}{(n+1)^2-1}\Big(\frac{I_{n+1}'(r)}{I_{n+1}(r)}-\frac{I_1'(r)}{I_1(r)}\Big)\hspace{2em} \text{for all $n\ge2$ and $r>0$};
\end{equation}
based on these two inequalities, we have, for $n\ge 2$,
\addtocounter{equation}{-1}
\renewcommand{\theequation}{\thesection.\arabic{equation}}
\begin{equation*}
    \mu_{n+1} > \mu_n.
\end{equation*}            
Moreover, the same proof from \cite{borisovich} can easily be modified to
establish 
\begin{equation*}
    \frac{1}{0^2-1}\Big(\frac{I_0'(r)}{I_0(r)}-\frac{I_1'(r)}{I_1(r)}\Big) > \frac{1}{2^2-1}\Big(\frac{I_2'(r)}{I_2(r)}-\frac{I_1'(r)}{I_1(r)}\Big),
\end{equation*}
together with \cite[(A.1)]{borisovich}, we also have
\begin{equation*}
    \mu_2 > \mu_0.
\end{equation*}
On the other hand, using $I_0'(r)=I_1(r)$ and $I_1'(r)=I_2(r) + \frac1r I_1(r)$, we  simplify $\mu_0$ as
\begin{equation*}
    \begin{split}
        \mu_0 &\,=\, \frac1{R_S}\Big[-1 + \frac{I_0(R_S)I_1(R_S)}{R_S I_0(R_S) I_2(R_S) + I_0(R_S)I_1(R_S) - R_S I_1^2(R_S)}\Big]\\
        &\,=\, \frac1{R_S}\Big[ -1 + 1\Big/\Big(R_S\frac{I_2(R_S)}{I_1(R_S)} - R_S\frac{I_1(R_S)}{I_0(R_S)} + 1 \Big) \Big].
    \end{split}
\end{equation*}
Since $I_0(R_S) I_2(R_S) < I_1^2(R_S)$, we get $\frac{I_2(R_S)}{I_1(R_S)} < \frac{I_1(R_S)}{I_0(R_S)}$, which implies $R_S\frac{I_2(R_S)}{I_1(R_S)} - R_S\frac{I_1(R_S)}{I_0(R_S)} + 1<1$. Moreover, by \cite[(A.1)]{borisovich}, $\frac{I_1'(R_S)}{I_1(R_S)}-\frac{I_0'(R_S)}{I_0(R_S)} = \frac{I_2(R_S)}{I_1(R_S)} - \frac{1}{R_S} - \frac{I_1(R_S)}{I_0(R_S)} > 0$, so that $R_S\frac{I_2(R_S)}{I_1(R_S)} - R_S\frac{I_1(R_S)}{I_0(R_S)} + 1 > 0$. Hence it is clear that $\mu_0>0$. Putting these estimates all together, we  have,
\begin{equation}
    \label{distinct}
    0<\mu_0 < \mu_2 < \mu_3 < \mu_4 < \cdots.
\end{equation}
With this monotonicity of $\mu_n$, we are now able to establish the bifurcation result for system \re{model2}.

We choose $X=X^{3+\alpha}_2$ and $Y=X^\alpha_2$ in the Crandall-Rabinowitz Theorem. The rigorous justifications of the 
Fr\'echet derivative and differentiability of $\mathcal{F}$ follow from the same
arguments as those in \cite{angio,Fengjie,FH3,Hongjing,Zejia,zhao2}. We now proceed to verify the
four assumptions of the theorem. The assumption (1) is naturally satisfied. Due to \re{distinct}, the kernel space satisfies
\begin{equation*}
    \text{Ker}[\mathcal{F}_{\tilde{R}}(0,\mu_n)] = \text{span}\{\cos(n\theta)\}\hspace{2em} \text{for even }n\ge2,
\end{equation*}
which indicates
\begin{equation*}
    \text{dim}(\text{Ker}[\mathcal{F}_{\tilde{R}}(0,\mu_n)])=1  \hspace{2em} \text{for even }n\ge2.
\end{equation*}
Moreover, since $\text{Im}[\mathcal{F}_{\tilde{R}}(0,\mu_n)]\oplus \text{span}\{\cos(n\theta)\} = Y$ is the whole space, we have $\text{codim}(\text{Im}[\mathcal{F}_{\tilde{R}}(0,\mu_n)])=1$. Finally, by differentiating \re{FreD1} with respect to $\mu$, we obtain
\begin{equation*}
    [\mathcal{F}_{\mu \tilde{R}} (0,\mu_n)] \cos(n\theta) = -\frac{I_1(R_S)}{I_0(R_S)}\Big(\frac{I_n'(R_S)}{I_n(R_S)}-\frac{I_1'(R_S)}{I_1(R_S)}\Big)\cos(n\theta) \notin \text{Im}[F_{\tilde{R}}(0,\mu_n)].
\end{equation*}
Thus all the assumptions in the Crandall-Rabinowitz Theorem are satisfied, and the following bifurcation result for system \re{model2} is established.

\begin{thm}
\label{bifur}
For each even $n\ge2$, $\mu=\mu_n(R_S)$ is a bifurcation point of the symmetry-breaking solution to the system \re{model2} with free boundary
\begin{equation}\label{form}
    r = R_s + \varepsilon \cos(n\theta) + o(\varepsilon), \hspace{2em}
    \mu = \mu(\varepsilon)= \mu_n(R_S) + o(1).
\end{equation}
\end{thm}

\begin{rem}\label{rem1}
The bifurcation result is actually valid for all $n\ge2$ not restricting to even $n$ only, however the proof is much more complicated.
\end{rem}

\section{The numerical method based on the neural network discretization}
\subsection{Boundary integral formulation}
Using the boundary integral formulation \cite{boundary1,boundary2,nonlinear1,nonlinear2}, we apply the standard representation formula \cite{friedman2005generalized} on system \re{model2} to obtain 
\begin{equation}
    \sigma(\xx) = \int_{\p \Omega}\bigg[ G_1(\xx,\y)\frac{\p \sigma(\y)}{\p {\bm n}_y}-\sigma(\y)\frac{\p G_1(\xx,\y)}{\p {\bm n}_y}\bigg] \dif S_y \hspace{2em} \text{for } \xx\in\Omega,\label{repre}
\end{equation}
where $G_1$ is the Green function for the operator $-\Delta+1$, namely, $G_1(\xx,\y)=G_1(|\xx-\y|)=\frac{i}{4}H_0^{(1)}(i|\xx-\y|)$ for two dimensional case, and $H_0^{(1)}$ is the Hankel function of the first kind. By using the ``jump'' relationship \cite{kress1989linear} as $\xx \rightarrow \p \Omega$, we derive
\begin{equation}
    \frac{\sigma(\xx)}{2} = \int_{\p \Omega}\bigg[G_1(\xx,\y)\frac{\p \sigma(\y)}{\p {\bm n}_y}-\sigma(\y)\frac{\p G_1(\xx,\y)}{\p {\bm n}_y}\bigg] \dif S_y \hspace{2em} \text{for } \xx\in\p \Omega.\label{jump}
\end{equation}
Combining with the boundary conditions in \re{model2}, we further obtain
\begin{equation}
    \frac{\mu + \kappa(\xx)}{2} = \int_{\p \Omega}\bigg[\beta G_1(\xx,\y)- (\mu + \kappa(\y))\frac{\p G_1(\xx,\y)}{\p {\bm n}_y}\bigg] \dif S_y \hspace{2em} \text{for } \xx\in\p \Omega.\label{repre1}
\end{equation}

Although $G_1(r)=O(|\ln r|)$ is weakly singular, $G'_r(r)=O(r^{-1})$  is strongly singular. To regularize the singularity of $G'_1(r)$, we introduce a new function $Q(r)$ defined as
\begin{equation}\label{Q}
    Q(r) = \frac1r \Big(G_1'(r) + \frac{1}{2\pi r}\Big).
\end{equation}
Then $Q(r)=O(|\ln r|), Q'(r)=O(r^{-1}), Q''(r)= O(r^{-2})$ (see \cite{hao2018convergence}).
As in \cite{hao2018convergence}, if we replace $G_1$ in \re{jump} by the fundamental solution for $-\Delta$ (which equals to $-\frac{1}{2\pi}\ln|\xx-\y|)$, and take $\sigma = 1$, we have
\begin{equation}\label{eq}
    \frac12 = -\int_{\p \Omega} \frac{\p}{\p {\bm n}_y}\Big(-\frac{1}{2\pi}\ln|\xx-\y|\Big)\dif S_y = \frac1{2\pi} \int_{\p \Omega}  \frac{\y-\xx}{|\xx-\y|^2}\cdot {\bm n}_y \dif S_y\hspace{2em}
\text{ for } x\in \p\Omega,\end{equation}
hence
\begin{equation}
    \label{eq2}
    \frac{\mu+\kappa(\xx)}{2} = \frac{1}{2\pi}\int_{\p \Omega} (\mu+\kappa(\xx)) \frac{\y-\xx}{|\xx-\y|^2}\cdot {\bm n}_y \dif S_y
    \hspace{2em}
\text{ for } x\in \p\Omega.
\end{equation}
Using \re{Q}, \re{eq}, as well as \re{eq2}, we can rewrite \re{repre1} as
\begin{equation}
    \label{repre2}
    \int_{\p \Omega}\bigg[\beta G_1(\xx,\y) - \Big((\mu + \kappa(\y))Q(|\xx-\y|) - \frac{\kappa(\y)-\kappa(\xx)}{2\pi|\xx-\y|^2}\Big)(\y-\xx)\cdot {\bm n}_y\bigg]\dif S_y = 0 \hspace{2em} \text{for } \xx\in\p\Omega,
\end{equation}
with $\p \Omega$ being the only unknown in the equation.

Equation \re{repre2} determines the free boundary $\p\Omega$. Due to the highly non-linear nature of \re{repre2}, there might be multiple solutions of $\p\Omega$. These solutions are expected to be computed by the machine learning techniques.

\subsection{The neural network discretization}
We use $\p \Omega: r=R(\theta),\;\theta \in (-\infty,\infty) $ to represent the unknown boundary. Clearly, it satisfies the $2\pi$-periodic boundary condition, namely, 
\begin{equation}
    \label{periodic}
    R(\theta) = R(\theta+2\pi).
\end{equation}
For simplicity, we shall restrict $\theta$ to $[0,2\pi]$.
By denoting in the polar coordinates, $$\xx=(R(\hat{\theta})\cos(\hat{\theta}),R(\hat{\theta})\sin(\hat{\theta}))\quad\text{ and } \quad\y=(R(\theta)\cos(\theta),R(\theta)\sin(\theta)),$$
we have 
\begin{eqnarray*}
 |\xx - \y| &=& \sqrt{[R(\hat{\theta})\cos(\hat{\theta})-R(\theta)\cos(\theta)]^2 + [R(\hat{\theta})\sin(\hat{\theta})-R(\theta)\sin(\theta)]^2}\\ 
 &=&  \sqrt{R^2(\hat{\theta}) + R^2(\theta) - 2R(\hat{\theta})R(\theta)\cos(\hat{\theta}-\theta)}\\
 &\triangleq& D[R].
\end{eqnarray*}
In addition,
\begin{eqnarray*}
        {\bm n}_y &=& \frac{1}{\sqrt{[R'(\theta)]^2 +  R^2(\theta)}} \Big(R'(\theta)\sin(\theta)+R(\theta)\cos(\theta),-R'(\theta)\cos(\theta)+R(\theta)\sin(\theta)\Big),\\\hbox{and~}\,
        \dif S_y &=& \sqrt{[R'(\theta)]^2 +  R^2(\theta)}\dif \theta.
\end{eqnarray*}
Using the mean-curvature formula in the 2-dimensional case for $r=R(\theta)$, we also have
\begin{equation}
\label{curvature}
    \kappa_R = \frac{R^2 + 2(R')^2 - R R''}{[R^2 + (R')^2]^\frac32}.
\end{equation}
For notational convenience, we denote $\kappa_R(\hat{\theta}) = \kappa_R(\xx)$ and $\kappa_R(\theta) = \kappa_R(\y)$. Both $\kappa_R(\hat{\theta})$ and $\kappa_R(\theta)$ can be computed by \re{curvature}.

Based on the above calculations, we rewrite the left-hand side of \re{repre2} as a functional of $R(\theta)$ as follows
\begin{equation}\label{L}
    \begin{split}
        \mathcal{L}[R] (\hat{\theta})&\triangleq  \int_{0}^{2\pi} \bigg[\beta G_1(D[R]) \sqrt{[R'(\theta)]^2 +  R^2(\theta)} - \Big((\mu+\kappa_R(\theta))Q(D[R]) - \frac{\kappa_R(\theta)-\kappa_R(\hat{\theta})}{2\pi (D[R])^2}\Big)\\
   &\hspace{11.5em} \Big(R^2(\theta)+R(\hat{\theta})R'(\theta)\sin(\hat{\theta}-\theta) - R(\hat{\theta})R(\theta)\cos(\hat{\theta}-\theta)\Big)\bigg]\dif \theta,
    \end{split}
\end{equation}
for $\hat{\theta}\in [0,2\pi]$. 
Equation \re{repre2} implies that $\mathcal{L}[R](\hat{\theta})\equiv 0$. To regularize the singularity of the kernel, we 
introduce a small constant $\tau>0$ in $D[R]$, namely, \begin{equation}\label{rf1}
    D_\tau[R] =  \sqrt{R^2(\hat{\theta}) + R^2(\theta) - 2R(\hat{\theta})R(\theta)\cos(\hat{\theta}-\theta) + \tau^2},
\end{equation}
and define the corresponding functional for each $\hat{\theta}\in [0,2\pi]$,
\begin{equation}
    \label{Lpolar}
    \begin{split}
    \mathcal{L}_\tau[R] (\hat{\theta})&\triangleq  \int_{0}^{2\pi} \bigg[\beta G_1(D_\tau[R]) \sqrt{[R'(\theta)]^2 +  R^2(\theta)} - \Big((\mu+\kappa_R(\theta))Q(D_\tau[R]) - \frac{\kappa_R(\theta)-\kappa_R(\hat{\theta})}{2\pi (D_\tau[R])^2}\Big)\\
   &\hspace{12em} \Big(R^2(\theta)+R(\hat{\theta})R'(\theta)\sin(\hat{\theta}-\theta) - R(\hat{\theta})R(\theta)\cos(\hat{\theta}-\theta)\Big)\bigg]\dif \theta,
    \end{split}
\end{equation}
Hence we recover \re{L} when $\tau\to0$ and have a non-singular kernel when $\tau>0$. In section 3, we shall prove that \re{Lpolar} is a good approximation of \re{L}.

Based on machine learning techniques, we write an approximation of the unknown free boundary function $R(\theta)$ by a single hidden layer neutral network:
\begin{equation}
R(\theta)\approx \sum_{i=1}^N a_i \Psi(b_i \theta + c_i) + d \triangleq \rho(\theta;\mathcal{X}),
\end{equation}
%
where $N$ is the width, $\mathcal{X} = (a_1,\cdots,a_N,b_1,\cdots,b_N,c_1,\cdots,c_N, d)\in \mathbb{R}^{3N+1}$ is the set of all the neural network's parameters, and $\Psi$ is a nonlinear ``activation'' function such as sigmoid function. Note that for each $\rho$, the operator $\mathcal{L}_\tau [\rho](\hat{\theta}_i)$ can be calculated analytically, where $\hat{\theta}_i$ are $m$ randomly sampled points, which are i.i.d. in $[0,2\pi]$. We consider the loss function:
\begin{equation}\label{loss}
    F(\mathcal{X}, \bm{\hat{\theta}}) \triangleq \frac1m \sum_{i=1}^m \Big(\mathcal{L}_\tau [\rho](\hat{\theta}_i)\Big)^2 + \sum_{\alpha=0}^2\Big(D^\alpha\big(\rho(0;\mathcal{X}) - \rho(2\pi;\mathcal{X})\big)\Big)^2.
\end{equation}
The reason why we choose up to second-order derivative in the second term is that the curvature term involves at most second-order derivative, hence we shall guarantee the continuity up to second-order. The loss function \re{loss} measures how well the function $\rho(\theta;\mathcal{X})$ satisfies equation \re{repre2} as well as the $2\pi$-periodic boundary condition \re{periodic}. Hence $\mathcal{X}$ is obtained via solving the following optimization problem
\begin{equation}
    \label{costf}
    \begin{split}
\min_{\mathcal{X}}~ {J}(\mathcal{X}) \triangleq \mathbb{E}_{ \bm{\hat{\theta}}}\big[F(\mathcal{X}, \bm{\hat{\theta}})\big] &= \mathbb{E}_{\hat{\theta}_i}\Big[\big(\mathcal{L}_\tau[\rho](\hat{\theta}_i)\big)^2\Big] \,+ \, \sum_{\alpha=0}^2\Big(D^\alpha \rho(0;\mathcal{X}) - D^\alpha \rho(2\pi;\mathcal{X}) \Big)^2 \\
&=  \int_0^{2\pi}\big(\mathcal{L}_\tau [\rho](\hat{\theta}_i)\big)^2  \nu( \hat{\theta}_i)\, \dif \hat{\theta}_i \,+ \, \sum_{\alpha=0}^2\Big(D^\alpha \rho(0;\mathcal{X}) - D^\alpha \rho(2\pi;\mathcal{X}) \Big)^2,
\end{split}
\end{equation}
where $\nu(\hat{\theta}_i)$ is a probability density of $\hat{\theta}_i \in[0,2\pi]$.



\subsection{Stochastic gradient descent training algorithm}
In order to solve \re{costf} numerically, we use the stochastic gradient descent method  shown in Algorithm \ref{alg}.
\begin{algorithm}[H]
\caption{}
\label{alg}
\begin{algorithmic}
\STATE {Choose an initial guess $\mathcal{X}_1$}
\FOR{$k=1,2,\cdots$}
\STATE Generate $m$ random points ${\bm \hat{\bm \theta}}_k = (\hat{\theta}_{k,i})_{i=1}^m$ from $[0,2\pi]$;\ 
\STATE Calculate the loss function at randomly sampled points $F(\mathcal{X}_k, {\bm \hat{\bm \theta}}_k)$;\
\STATE Compute a stochastic vector $G(\mathcal{X}_k, {\bm \hat{\bm \theta}}_k) = \nabla_\mathcal{X} F(\mathcal{X}_k,{\bm \hat{\bm \theta}}_k)$;\
\STATE Set the new iterate as $\mathcal{X}_{k+1} = \mathcal{X}_k - \alpha_n G(\mathcal{X}_k,{\bm \hat{\bm \theta}}_k)$;
\ENDFOR
\end{algorithmic}
\end{algorithm}

Due to the non-convexity of $J(\mathcal{X})$,
 $\mathcal{X}_k$ may  stuck at a local minimum (not a global minimum). Nevertheless, stochastic gradient descent has proven very effective in training deep learning models to obtain the global minimum.

\section{Convergence of the neural network discretization}
In this section, we shall prove that the numerical solution with the  neural network discretization converges to the unknown boundary of system \re{model2} as the number of hidden units tends to infinity, namely, 
$$\text{there exists $\rho\in \mathcal{C}^N$ such that $J(\rho(\theta;\mathcal{X}))\rightarrow 0$ as $n\rightarrow \infty$};$$
where
\begin{equation}\label{set}
    \mathcal{C}^N: \; \{\rho^N(\theta): [0,2\pi]\rightarrow \mathbb{R} \, | \, \rho^N(\theta) = \sum_{i=1}^N a_i \Psi(b_i \theta + c_i) + d\}.
\end{equation}
The precise statement is included in Theorem \ref{main1}.

\subsection{Preliminary estimates}
Denote 
\begin{eqnarray*}
    h(\xx) &=& \int_{\p \Omega} \beta G_1(D[R]) \dif S_y,\\
    g(\xx) &=& \int_{\p \Omega} (\mu + \kappa_R(\y)) Q(D[R])(\y - \xx)\cdot {\bm n}_y \dif S_y,\\
    w(\xx) &=& \int_{\p \Omega} \frac{\kappa_R(\y) - \kappa_R(\xx)}{2\pi} \frac{\y - \xx}{(D[R])^2}\cdot {\bm n}_y \dif S_y.
\end{eqnarray*}
Then the operator $\mathcal{L}$ in \re{L} can be separated into three parts, namely,
\begin{equation*}
    \mathcal{L}[R](\xx) =  h(\xx) - g(\xx) + w(\xx).
\end{equation*}
As mentioned before, since $G_1(\xx,\y) = G_1(|\xx-\y|) = \frac{i}{4}H_0^{(1)}(i|\xx-\y|)$ is singular at $|\xx-\y|=0$, we further introduce,
\begin{eqnarray*}
    h_\tau(\xx) &=& \int_{\p \Omega} \beta G_1(D_\tau[R]) \dif S_y,\\
    g_\tau(\xx) &=& \int_{\p \Omega} (\mu + \kappa_R(\y)) Q(D_\tau[R])(\y - \xx)\cdot {\bm n}_y \dif S_y,\\
    w_\tau(\xx) &=& \int_{\p \Omega} \frac{\kappa_R(\y) - \kappa_R(\xx)}{2\pi} \frac{\y - \xx}{(D_\tau[R])^2}\cdot {\bm n}_y \dif S_y.
\end{eqnarray*}
Correspondingly, $\mathcal{L}_\tau$ in \re{Lpolar} is also separated into three pieces:
\begin{equation*}
    \mathcal{L}_\tau[R](\xx) =  h_\tau(\xx) - g_\tau(\xx) + w_\tau(\xx). 
\end{equation*}

By \textit{Lemmas 7.6 -- 7.8} in \cite{hao2018convergence}, we 
have the following results:
\begin{lem}\label{approx}
Suppose that $\p \Omega: r=R(\theta) \in C^{3}(-\infty,\infty)$
and $2\pi$-periodic, then
\begin{gather}
    \|h-h_\tau\|_{L^\infty} \le C \tau |\ln \tau|,\label{tau1}\\
    \|g-g_\tau\|_{L^\infty} \le C \|\kappa_R\|_{L^\infty}\tau  \le C\tau,\label{tau2}\\
    \|w-w_\tau\|_{L^\infty}  \le C \|\kappa_R\|_{C^1}\tau  \le C \tau,\label{tau3},
\end{gather}
where the constant $C$ is independent of $\tau$. 
\end{lem}
\begin{proof}

The proof is really lengthy, here we only point out some key steps. Our definitions of $h(\xx)$ and $h_\tau(\xx)$ are equivalent to $h(\hat{s})$ and $h_\varepsilon(\hat{s})$ when $f(s)=\beta$ \cite[(93)]{hao2018convergence} (Note that functions in \cite{hao2018convergence} are defined based on curve length $s$. Although they look different from our definitions, they are actually equivalent to our definitions of functions. The curve length parameters $s$ and $\hat{s}$ correspond to the parameters $\y$ and $\xx$ here.) Hence \re{tau1} directly follows from \textit{Lemma 7.6}. Similarly, $g(\xx)$ and $g_\tau(\xx)$ are equivalent to $g_2(\hat{s})$ and $g_{2\varepsilon}(\hat{s})$ when $f(s)=\mu+\kappa_R(\y)$ \cite[pg.\,146]{hao2018convergence}; $w(\xx)$ and $w_\tau(\xx)$ are equivalent to $-g_{11}(\hat{s})$ and $-g_{11\varepsilon}(\hat{s})$ when $f(s) = \kappa_R(\y)$ and $f(\hat{s})=\kappa_R(\xx)$ \cite[(110), (111)]{hao2018convergence}. By \textit{Lemma 7.7 and 7.8}, we shall get estimates \re{tau2} and \re{tau3}. Notice that we need $\|\kappa_R\|_{C^1}$ in \re{tau3}, and $\kappa_R$ involves at most second-order derivatives of $R(\theta)$, hence we require $\p \Omega: r=R(\theta) \in C^{3}$.
\end{proof}

\addtocounter{equation}{-1}
\renewcommand{\theequation}{\thesection.\arabic{equation}}

Based on Lemma \ref{approx}, we shall have
\begin{equation*}
    \|(\mathcal{L}_\tau-\mathcal{L})[R]\|_{L^\infty} \le  \|h-h_\tau\|_{L^\infty} + \|g-g_\tau\|_{L^\infty} + \|w - w_\tau\|_{L^\infty} \le C\tau |\ln \tau| + C\tau,
\end{equation*}
which indicates that \re{Lpolar} is a good approximation of \re{L}. Recall that it follows from \re{repre2} that $\mathcal{L}[R]\equiv 0$, then we immediately derive
\begin{equation}
    \label{tau}
    \|\mathcal{L}_\tau[R]\|_{L^\infty} \le C\tau |\ln \tau| + C\tau.
\end{equation}

\subsection{The neural network approximation}
By \textit{Theorem 3} of \cite{hornik1991approximation} we know that if the activation function $\Psi\in C^3(\mathbb{R})$ is nonconstant and bounded, then the space $\cup_{n=1}^\infty \mathcal{C}^n$  is uniformly 3-dense on compacta in $C^3(\mathbb{R})$. This means that for $R(\theta) \in C^3(-\infty,\infty)$ and every $0<\delta<1$, there is $\rho(\theta;\mathcal{X})\in \cup_{N=1}^\infty \mathcal{C}^N$ such that
\begin{equation}
    \label{eps}
    \|\rho-R\|_{3,[0,2\pi]} \le \delta,
\end{equation}
where $\|f\|_{3,[0,2\pi]} := \max_{\alpha\le 3} \sup_{x\in [0,2\pi]} |D^\alpha f(x)|$. Clearly, \re{eps2} implies
\begin{equation}
    \label{eps2}
    \|\rho-R\|_{L^\infty([0,2\pi])}, \|\rho' - R'\|_{L^\infty([0,2\pi])}, \|\rho''- R''\|_{L^\infty([0,2\pi])}, \|\rho'''-R'''\|_{L^\infty([0,2\pi])} \le \delta.
\end{equation}
Based on \re{eps2}, let's first bound $\|\mathcal{L}_\tau [\rho]-\mathcal{L}_\tau[R]\|_{L^\infty}$, which is a key estimate in proving the convergence theorem. Throughout the rest of this paper, $C$ is used to represent a generic constant independent of $\tau$ and $\delta$, which might change from a line to next.

Recall the formulas for $\mathcal{L}_\tau[R]$ and $D_\tau[R]$ in \re{Lpolar} and \re{rf1}, respectively:
\begin{equation*}
    \begin{split}
   \mathcal{L}_\tau[R] (\hat{\theta})&=  \int_{0}^{2\pi} \bigg[\beta G_1(D_\tau[R]) \sqrt{[R'(\theta)]^2 +  R^2(\theta)} - \Big((\mu+\kappa_R(\theta))Q(D_\tau[R]) - \frac{\kappa_R(\theta)-\kappa_R(\hat{\theta})}{2\pi (D_\tau[R])^2}\Big)\\
   &\hspace{12em} \Big(R^2(\theta)+R(\hat{\theta})R'(\theta)\sin(\hat{\theta}-\theta) - R(\hat{\theta})R(\theta)\cos(\hat{\theta}-\theta)\Big)\bigg]\dif \theta,\\
   D_\tau[R] &=  \sqrt{R^2(\hat{\theta}) + R^2(\theta) - 2R(\hat{\theta})R(\theta)\cos(\hat{\theta}-\theta) + \tau^2}.
    \end{split}
\end{equation*}
Correspondingly, $\mathcal{L}_\tau[\rho]$ takes the following form
\begin{equation}\label{Ltauf}
    \begin{split}
     \mathcal{L}_\tau[\rho] (\hat{\theta})&=  \int_{0}^{2\pi} \bigg[\beta G_1(D_\tau[\rho]) \sqrt{[\rho'(\theta)]^2 +  \rho^2(\theta)} - \Big((\mu+\kappa_\rho(\theta))Q(D_\tau[\rho]) - \frac{\kappa_\rho(\theta)-\kappa_\rho(\hat{\theta})}{2\pi (D_\tau[\rho])^2}\Big)\\
   &\hspace{12em} \Big(\rho^2(\theta)+\rho(\hat{\theta})\rho'(\theta)\sin(\hat{\theta}-\theta) - \rho(\hat{\theta})\rho(\theta)\cos(\hat{\theta}-\theta)\Big)\bigg]\dif \theta,
    \end{split}
\end{equation}
where
\begin{equation}\label{Dtauf}
    D_\tau[\rho] =  \sqrt{\rho^2(\hat{\theta}) + \rho^2(\theta) - 2\rho(\hat{\theta})\rho(\theta)\cos(\hat{\theta}-\theta) + \tau^2}.
\end{equation}
Notice that we use $\kappa_R$ and $\kappa_\rho$ to differentiate the curvature on different curves. By \re{curvature},
\begin{equation}
    \label{kappa12}
    \kappa_R =  \frac{R^2 + 2(R')^2 - R R''}{[R^2 + (R')^2]^\frac32}, \hspace{2em}\text{and  \;} \kappa_\rho = \frac{\rho^2 + 2(\rho')^2 - \rho \rho''}{[\rho^2 + (\rho')^2]^\frac32}.
\end{equation}
Subtracting $\mathcal{L}_\tau[\rho]$ from $\mathcal{L}_\tau[R]$, we derive, for each $\hat{\theta}\in [0,2\pi]$,
\begin{equation*}
    \Big|\mathcal{L}_\tau [\rho](\hat{\theta})-\mathcal{L}_\tau[R](\hat{\theta})\Big| \le \beta \int_0^{2\pi} \Big|\text{\Rmnum{1}}\Big| \dif \theta + \int_0^{2\pi} \Big|\text{\Rmnum{2}}\Big| \dif \theta + \int_0^{2\pi} \Big|\text{\Rmnum{3}}\Big| \dif \theta,
\end{equation*}
where
\begin{equation*}
    \begin{split}
        \text{\Rmnum{1}} &= G_1(D_\tau[\rho])\sqrt{[\rho'(\theta)]^2 +  \rho^2(\theta)} - G_1(D_\tau[R])\sqrt{[R'(\theta)]^2 +  R^2(\theta)},\\
        \text{\Rmnum{2}} &= Q(D_\tau[\rho])\big(\mu + \kappa_\rho(\theta)\big)\big(\rho^2(\theta)
   +\rho(\hat{\theta})\rho'(\theta)\sin(\hat{\theta}-\theta) - \rho(\hat{\theta})\rho(\theta)\cos(\hat{\theta}-\theta)\big)\\
   &\hspace{2em}- Q(D_\tau[R])\big(\mu+\kappa_R(\theta)\big)\big(R^2(\theta)
   +R(\hat{\theta})R'(\theta)\sin(\hat{\theta}-\theta) - R(\hat{\theta})R(\theta)\cos(\hat{\theta}-\theta)\big),\\
   \text{\Rmnum{3}} &=\frac{\kappa_\rho(\theta)- \kappa_\rho(\hat{\theta})}{2\pi (D_\tau[\rho])^2}\big(\rho^2(\theta)
   +\rho(\hat{\theta})\rho'(\theta)\sin(\hat{\theta}-\theta) - \rho(\hat{\theta})\rho(\theta)\cos(\hat{\theta}-\theta)\big)\\
   &\hspace{2em}-\frac{\kappa_R(\theta)-\kappa_R(\hat{\theta})}{2\pi (D_\tau[R])^2} \big(R^2(\theta)
   +R(\hat{\theta})R'(\theta)\sin(\hat{\theta}-\theta) - R(\hat{\theta})R(\theta)\cos(\hat{\theta}-\theta)\big).
    \end{split}
\end{equation*}

In order to estimate $\big|\mathcal{L}_\tau [\rho](\hat{\theta})-\mathcal{L}_\tau[R](\hat{\theta})\big|$, we need to estimate $|\text{\Rmnum{1}}|$,  $|\text{\Rmnum{2}}|$, and $|\text{\Rmnum{3}}|$, respectively.  For term \Rmnum{1}, we insert a term $G_1(D_\tau[\rho])\sqrt{[R'(\theta)]^2 +  R^2(\theta)}$ and subtract the same term; after rearranging the terms in \Rmnum{1}, we obtain
\begin{equation}\label{term1}
    \begin{split}
        \Big|\text{\Rmnum{1}}\Big| \le&\; \Big|G_1(D_\tau[\rho])\Big| \Big|\sqrt{[\rho'(\theta)]^2 +  \rho^2(\theta)} - \sqrt{[R'(\theta)]^2 +  R^2(\theta)} \Big|\\
        &\hspace{3em}+ \Big|G_1(D_\tau[\rho]) - G_1(D_\tau[R])\Big|\sqrt{[R'(\theta)]^2 +  R^2(\theta)}.
    \end{split}
\end{equation}
Using the inequality $|\sqrt{x}-\sqrt{y}| \le \sqrt{|x-y|}$, and combining with estimates \re{eps2}, we have
\begin{equation}\label{1}
\begin{split}
    \Big|\sqrt{[\rho'(\theta)]^2 +  \rho^2(\theta)} - \sqrt{[R'(\theta)]^2 +  R^2(\theta)} \Big| \;&\le\; \sqrt{\Big|[\rho'(\theta)]^2 +  \rho^2(\theta)-[R'(\theta)]^2 -  R^2(\theta) \Big|}\\
    &\le\; \sqrt{\Big|[\rho'(\theta)]^2 - [R'(\theta)]^2\Big| + \Big|\rho^2(\theta) - R^2(\theta)\Big|} \\
    &\le\; C\delta^{\frac12}.
    \end{split}
\end{equation}
In a similar manner, $|D_\tau[\rho] - D_\tau[R]|$ can also be estimated:
\begin{equation}\label{2}
    \begin{split}
        \Big|D_\tau[\rho] - D_\tau[R]\Big| &\le \; \sqrt{\Big|\rho^2(\hat{\theta}) + \rho^2(\theta) -R^2(\hat{\theta}) - R^2(\theta) - 2\big(\rho(\hat{\theta})\rho(\theta)-R(\hat{\theta})R(\theta)\big)\cos(\hat{\theta}-\theta)\Big|}\\
        &\le\; \sqrt{\Big|\rho^2(\hat{\theta}) - R^2(\hat{\theta})\Big| + \Big|\rho^2(\theta) - R^2(\theta)\Big| + 2\Big|\rho(\hat{\theta})\rho(\theta)-R(\hat{\theta})R(\theta)\Big|}\\
        &\le\; C\delta^{\frac12}.
    \end{split}
\end{equation}
In addition, since $R\in C^2$, we clearly have
\begin{equation}
    \label{3} \sqrt{[R'(\theta)]^2 +  R^2(\theta)} \le C.
\end{equation}
Next we proceed to consider the terms involving $G_1$. We compute 
$$G_1(r) =\frac i4 H_0^{(1)}(ir),\hspace{2em}\text{and }\hspace{2em} G_1'(r)=-\frac14(H_0^{(1)})'(ir)=\frac14 H_1^{(1)}(ir);$$
and it is clear that both $G_1(r)$ and $G_1'(r)$ are singular at $r=0$. Here we collect some formulas for the Hankel function, and we focus on the approximation when $r\rightarrow 0$:
\begin{equation*}
    \begin{split}
        &H_0^{(1)}(z) = J_0(z) + i Y_0(z),\hspace{2em} H_1^{(1)}(z) = J_1(z) + i Y_1(z),\hspace{2em} \text{([9.1.3] of \cite{abramowitz1948handbook})},\\
        &J_0(z) = \phi_1(-z^2), \hspace{2em} \phi_1(0)=1,\hspace{2em} \phi_1\in C^\infty, \hspace{2em} \text{([9.1.12] of \cite{abramowitz1948handbook})},\\
        &Y_0(z) = \frac2\pi \Big[\ln\Big(\frac12 z\Big)+\gamma\Big]J_0(z) + \phi_2(z^2),\hspace{2em} \phi_2\in C^\infty, \hspace{2em}\text{([9.1.13] of \cite{abramowitz1948handbook})},\\
        &J_1(z) = \frac12 z \phi_3(-z^2),\hspace{2em} \phi_3(0)=1,\hspace{2em} \phi_3\in C^\infty, \hspace{2em} \text{([9.1.10] of \cite{abramowitz1948handbook})},\\
        &Y_1(z) = -\frac{2}{\pi z} + \frac{2}{\pi}\ln\Big(\frac12 z\Big) J_1(z)-\frac{z}{2\pi}\phi_4(-z^2),\hspace{2em} \phi_4\in C^\infty, \hspace{2em} \text{([9.1.11] of \cite{abramowitz1948handbook})}.
    \end{split}
\end{equation*}
It follows that
\begin{gather*}
     G_1(r) =\frac i4 H_0^{(1)}(ir) = -\frac{1}{2\pi}\ln r\cdot \phi_1(r^2) + \phi_5(r^2),\hspace{2em} \phi_5\in C^\infty;\\
     G_1'(r) = \frac14 H_1^{(1)}(ir) = -\frac{1}{2\pi r}-\frac{1}{4\pi}r\ln r\cdot \phi_3(r^2) + r\phi_6(r^2),\hspace{2em} \phi_6\in C^\infty.
\end{gather*}
Since $\tau > 0$, both $D_\tau[\rho]$ and $D_\tau[R]$ are greater than $\tau$; hence there exists a constant $C$ which is independent of $\tau$ and $\delta$ such that
\begin{gather*}
        \Big|G_1(D_\tau[\rho])\Big| \le C|\ln \tau|,\\
        \Big|G_1(D_\tau[\rho]) - G_1(D_\tau[R])\Big| \le \frac{C}{\tau} \Big|D_\tau[\rho] - D_\tau[R]\Big|.
\end{gather*}
Substituting the above two inequalities, and estimates \re{1}, \re{2} as well as \re{3} all into \re{term1}, we finally derive
\begin{equation}
    \label{one}
    |\text{\Rmnum{1}}| \le C|\ln \tau|\delta^{\frac12} + \frac{C}{\tau} \delta^{\frac12},
\end{equation}
where the constant $C$ is independent of $\tau$ and $\delta$.

After we show the bound for \Rmnum{1}, we can estimate \Rmnum{2} and \Rmnum{3} in the same manner. Recall that
\begin{equation*}
    Q(r) = \frac1r\Big(G_1'(r) + \frac{1}{2\pi r}\Big),
\end{equation*}
then
\begin{gather*}
    Q(r) = -\frac{1}{4\pi}\ln r\cdot \phi_3(r^2) + \phi_6(r^2),\hspace{2em} Q'(r)=O(r^{-1}).
\end{gather*}
Therefore, similar as function $G_1$, there exists a constant $C$ not independent of $\tau$ and $\delta$ such that
\begin{gather*}
        \Big|Q(D_\tau[\rho])\Big| \le C|\ln \tau|,\\
        \Big|Q(D_\tau[\rho]) - Q(D_\tau[R])\Big| \le \frac{C}{\tau} \Big|D_\tau[\rho] - D_\tau[R]\Big|.
\end{gather*}
Next we turn our attention to the terms involving the curvature. Since $R,\rho\in C^3$, $\kappa_R$ and $\kappa_\rho$ are both bounded based on \re{kappa12} (note that we are only interested in the solutions which are away from the origin, hence the denominator of the curvature is not close to 0). Moreover, subtracting $\kappa_\rho$ from $\kappa_R$, recalling also \re{1} and \re{eps2}, we have 
\begin{equation*}
\begin{split}
    \Big|\kappa_\rho - \kappa_R\Big| \;&\le\; \Big|\frac{\rho^2 + 2(\rho')^2 - \rho \rho''}{[\rho^2 + (\rho')^2]^\frac32} - \frac{\rho^2 + 2(\rho')^2 - \rho \rho''}{[R^2 + (R')^2]^\frac32}\Big| + \Big|\frac{\rho^2 + 2(\rho')^2 - \rho \rho''}{[R^2 + (R')^2]^\frac32} - \frac{R^2 + 2(R')^2 - R R''}{[R^2 + (R')^2]^\frac32}\Big|\\
    &\le \; C\Big|\big(\sqrt{\rho^2 + (\rho')^2}\,\big)^3 - \big(\sqrt{R^2 + (R')^2}\,\big)^3\Big| + C\Big|\rho^2-R^2 + 2(\rho')^2 - 2(R')^2 - \rho\rho'' + R R'' \Big|\\
    &\le \; C\big|\sqrt{\rho^2 + (\rho')^2} - \sqrt{R^2 + (R')^2}\big| + C\delta\\
    &\le \; C\delta^{\frac12} + C\delta \; \le \; C\delta^{\frac12}.
    \end{split}
\end{equation*}

Therefore, we  derive 
\begin{equation}\label{two}
    \begin{split}
        \Big|\text{\Rmnum{2}}\Big| \le&\; \Big|Q(D_\tau[\rho])\Big|\Big|\mu + \kappa_\rho \Big|\Big|\rho^2(\theta) - R^2(\theta)
   +\big(\rho(\hat{\theta})\rho'(\theta)-R(\hat{\theta})R'(\theta)\big)\sin(\hat{\theta}-\theta)\\
   &\hspace{17.2em}- \big(\rho(\hat{\theta})\rho(\theta)-R(\hat{\theta})R(\theta)\big)\cos(\hat{\theta}-\theta)\Big| \\
   &\,+ \Big|Q(D_\tau[\rho])\Big| \Big|\kappa_\rho -\kappa_R\Big| \Big|R^2(\theta) +R(\hat{\theta})R'(\theta)\sin(\hat{\theta}-\theta) - R(\hat{\theta})R(\theta)\cos(\hat{\theta}-\theta)\Big|\\
   &\,+ \Big|Q(D_\tau[\rho]) - Q(D_\tau[R])\Big| \Big|\mu + \kappa_R\Big| \Big|R^2(\theta) +R(\hat{\theta})R'(\theta)\sin(\hat{\theta}-\theta) - R(\hat{\theta})R(\theta)\cos(\hat{\theta}-\theta)\Big|\\
   \le&\; C|\ln \tau|\delta + C|\ln \tau|\delta^{\frac12} +  \frac{C}{\tau}\delta^{\frac12} \;\le\; C|\ln \tau|\delta^{\frac12} +  \frac{C}{\tau}\delta^{\frac12};
    \end{split}
\end{equation}

\begin{equation}\label{three}
    \begin{split}
        \Big|\text{\Rmnum{3}}\Big| \le&\; \Big|\frac{\kappa_\rho(\theta)- \kappa_\rho(\hat{\theta})}{2\pi (D_\tau[\rho])^2}\Big| \Big|\rho^2(\theta) - R^2(\theta)
   +\big(\rho(\hat{\theta})\rho'(\theta)-R(\hat{\theta})R'(\theta)\big)\sin(\hat{\theta}-\theta)\\
   &\hspace{17.2em}- \big(\rho(\hat{\theta})\rho(\theta)-R(\hat{\theta})R(\theta)\big)\cos(\hat{\theta}-\theta)\Big| \\
   &\, +  \Big|\frac{\kappa_\rho(\theta)- \kappa_\rho(\hat{\theta})}{2\pi(D_\tau [\rho])^2}-\frac{\kappa_\rho(\theta)- \kappa_\rho(\hat{\theta})}{2\pi(D_\tau[R])^2}\Big|\Big|R^2(\theta) +R(\hat{\theta})R'(\theta)\sin(\hat{\theta}-\theta) - R(\hat{\theta})R(\theta)\cos(\hat{\theta}-\theta)\Big|\\
   &\, + \Big|\frac{\kappa_\rho(\theta) - \kappa_R(\theta)-\kappa_\rho(\hat{\theta})+\kappa_R(\hat{\theta})}{2\pi(D_\tau[R])^2} \Big|\Big|R^2(\theta) +R(\hat{\theta})R'(\theta)\sin(\hat{\theta}-\theta) - R(\hat{\theta})R(\theta)\cos(\hat{\theta}-\theta)\Big|\\
   \le&\; \frac{C}{\tau^2}\delta + \frac{C}{\tau^3}\delta^{\frac12} + \frac{C}{\tau^2}\delta^{\frac12} \;\le \; \frac{C}{\tau^3}\delta^{\frac12}.
    \end{split}
\end{equation}

Now we are able to estimate $\big|\mathcal{L}_\tau [\rho](\hat{\theta})-\mathcal{L}_\tau[R](\hat{\theta})\big|$. Recall that
\begin{equation*}
    \big|\mathcal{L}_\tau [\rho](\hat{\theta}) -\mathcal{L}_\tau[R](\hat{\theta})\big| \le \beta \int_0^{2\pi} \big|\text{\Rmnum{1}}\big| \dif \theta + \int_0^{2\pi} \big|\text{\Rmnum{2}}\big| \dif \theta + \int_0^{2\pi} \big|\text{\Rmnum{3}}\big| \dif \theta;
\end{equation*}
together with \re{one}, \re{two} and \re{three}, it follows that, for each $\hat{\theta}\in [0,2\pi]$,
\begin{equation*}
\begin{split}
    \big|\mathcal{L}_\tau [\rho](\hat{\theta}) -\mathcal{L}_\tau[R](\hat{\theta})\big| &\le 2\beta\pi \max|\text{\Rmnum{1}}| + 2\pi \max|\text{\Rmnum{2}}| + 2\pi \max|\text{\Rmnum{3}}|\\
    &\le 2\beta\pi\Big(C|\ln \tau|\delta^{\frac12} + \frac{C}{\tau} \delta^{\frac12}\Big) + 2\pi \Big(C|\ln \tau|\delta^{\frac12} + \frac{C}{\tau}\delta^{\frac12}\Big) + 2\pi\frac{C}{\tau^3}\delta^{\frac12}\\
    &\le C|\ln \tau|\delta^{\frac12} + \frac{C}{\tau^3} \delta^{\frac12},
    \end{split}
\end{equation*}
which is equivalent to
\begin{equation}
    \label{taufr}
    \|\mathcal{L}_\tau [\rho]-\mathcal{L}_\tau[R]\|_{L^\infty} \le C|\ln \tau|\delta^{\frac12} + \frac{C}{\tau^3} \delta^{\frac12}.
\end{equation}

\subsection{Convergence theorem}

Here is our convergence theorem:
\begin{thm}
\label{main1}
Let $\mathcal{C}^n(\Psi)$ be given by \re{set} where $\Psi$ is assumed to be in $C^3(\mathbb{R})$, bounded and non-constant. Then for every $0<\delta < 1$, there exists $\rho(\theta;\mathcal{X})\in \cup_{N=1}^\infty \mathcal{C}^N$ and a positive constant $K$ such that
\begin{equation*}
    J(\rho(\theta;\mathcal{X})) \le K \delta^{\frac18},
\end{equation*}
where the constant $K$ does not depend upon $\delta$.  
\end{thm}
\begin{proof}
We start from the definition of $J$ in \re{costf}. Notice that $R$ satisfies the $2\pi$-periodic boundary condition, i.e., $R(0) = R(2\pi)$, $R'(0) = R'(2\pi)$, $R''(0)=R''(2\pi)$. Therefore, we have
\begin{equation*}
    \begin{split}
    J(\rho(\theta;\mathcal{X})) \;=& \int_0^{2\pi}\big(\mathcal{L}_\tau [\rho](\hat{\theta}_i)\big)^2  \nu( \hat{\theta}_i)\, \dif \hat{\theta}_i \,+ \, \sum_{\alpha=0}^2\Big(D^\alpha\big(\rho(0;\mathcal{X}) - \rho(2\pi;\mathcal{X})\big)\Big)^2\\
    =& \int_0^{2\pi} \big(\mathcal{L}_\tau[\rho](\hat{\theta}_i) -\mathcal{L}_\tau[R](\hat{\theta}_i) + \mathcal{L}_\tau[R](\hat{\theta}_i)\big)^2 \nu(\hat{\theta}_i)\, \dif \hat{\theta}_i \\
    &\hspace{12em}+ \sum_{\alpha=0}^2\Big(D^\alpha\big(\rho(0;\mathcal{X})-R(0) - \rho(2\pi;\mathcal{X})+R(2\pi)\big)\Big)^2\\
   \le&\; 2\int_0^{2\pi} \big(\mathcal{L}_\tau[\rho](\hat{\theta}_i)-\mathcal{L}_\tau[R](\hat{\theta}_i)\big)^2 \nu(\hat{\theta}_i)\, \dif \hat{\theta}_i \,+\, 2\int_0^{2\pi}\big(\mathcal{L}_\tau[R](\hat{\theta}_i)\big)^2 \nu(\hat{\theta}_i)\, \dif \hat{\theta}_i \\
    &\hspace{12em}+ 2\sum_{\alpha=0}^2\Big(\big(\rho(0;\mathcal{X})-R(0)\big)\Big)^2 + 2\sum_{\alpha=0}^2\Big( \big(\rho(2\pi;\mathcal{X})-R(2\pi)\big)\Big)^2\\
    \le&\; 2\,\|\mathcal{L}_\tau [\rho]-\mathcal{L}_\tau[R]\|_{L^\infty}^2 + 2\,\|\mathcal{L}_\tau[R]\|_{L^\infty}^2 + 12\,\|\rho-R\|_{2,[0,2\pi]}^2.
    \end{split}
\end{equation*}
We combine it with estimates \re{tau}, \re{eps2} as well as \re{taufr}, and use the fact that $(a+b)^2 \le 2(a^2+b^2)$, to obtain
\begin{equation*}
    J(\rho(\theta;\mathcal{X})) \;\le\; C\Big(|\ln \tau|^2 \delta + \frac{\delta}{\tau^6} + |\ln \tau|^2 \tau^2 +  \tau^2 +  \delta^2 \Big).
\end{equation*}
Take $\tau=\delta^{\frac18}$, we have
\begin{equation*}
\begin{split}
    J(\rho(\theta;\mathcal{X})) &\le C\Big(|\ln \delta^{\frac18}|^2 \delta + \delta^{\frac14} + |\ln \delta^{\frac18}|^2 \delta^{\frac14} + \delta^{\frac14} +  \delta^2 \Big)\le C\Big(|\ln \delta^{\frac18}|^2 \delta^{\frac14} +\delta^{\frac14} \Big).
    \end{split}
\end{equation*}
It is easy to show, for $0<x<1$,
$$|\ln x|^2 x^2 < x;$$
hence when $0<\delta <1$,
$$|\ln \delta^{\frac18}|^2 \delta^{\frac14} \le \delta^{\frac18}.$$
By choosing $K=2C$, we finally obtain
\begin{equation*}
    J(f) \le C\Big(\delta^{\frac18} + \delta^{\frac14}\Big)\le 2\,C \delta^{\frac18} \le K\delta^{\frac18}
\end{equation*}
which completes the proof.
\end{proof}

\section{Convergence for the Stochastic Gradient Descent}
In this section, we will prove the convergence of the stochastic gradient descent. Recall the loss function we consider in \re{loss}, 
\begin{equation}\label{F1}
    F(\mathcal{X}, \bm{\hat{\theta}}) = \frac1m \sum_{i=1}^m \Big(\mathcal{L}_\tau [\rho](\hat{\theta}_i)\Big)^2 + \sum_{\alpha=0}^2\Big(D^\alpha \rho(0;\mathcal{X}) - D^\alpha \rho(2\pi;\mathcal{X}) \Big)^2,
\end{equation}
and 
\begin{eqnarray}
J(\mathcal{X}) &=& \mathbb{E}_{\bm{\hat{\theta}}}[ F(\mathcal{X},\bm{\hat{\theta}})]\label{J},\\
G(\mathcal{X},\bm{\hat{\theta}}) &=& \nabla_\mathcal{X} F(\mathcal{X},\bm{\hat{\theta}}).\label{G}
\end{eqnarray}



We have the following convergence result for the stochastic gradient descent:
\begin{thm}
\label{main2}
Assume further the activation function $\Psi$ is $C^4(\mathbb{R})$, and $\{\mathcal{X}_k\}$ is contained in a bounded open set, let the stochastic gradient descent method (Algorithm \ref{alg}) run with a stepsize sequence satisfying
\begin{equation*}
    \sum\limits_{k=1}^\infty \alpha_k = \infty \;\text{  and  }\; \sum\limits_{k=1}^\infty \alpha_k^2 < \infty,
\end{equation*}
then, with $A_k:= \sum_{k=1}^K \alpha_k$, the following convergence is obtained,
\begin{equation*}
    \mathbb{E}\Big[\frac{1}{A_k} \sum_{k=1}^K \alpha_k \|\nabla F(\mathcal{X})\|_2^2\Big] \rightarrow 0 \;\text{ as }\; K\rightarrow \infty.
\end{equation*}
\end{thm}
\begin{proof}
To proof Theorem \ref{main2}, it suffices to verify the assumptions 4.1 and 4.3 in \cite{optimization}:

\subsection{Lipschitz-continuous objective gradients}
First, let's prove that the objective function $J(\mathcal{X})$ is continuously differentiable and the gradient function of $J$, namely, $\nabla_\mathcal{X} J$ is Lipschitz continuous (cf., \cite{optimization}). Since
\begin{equation}\label{FJ}
    \nabla_\mathcal{X} J(\mathcal{X}) = \nabla_\mathcal{X} \mathbb{E}_{\bm{\hat{\theta}}} [F(\mathcal{X},{\bm{\hat{\theta}}})] = \mathbb{E}_{\bm{\hat{\theta}}}[\nabla_\mathcal{X} F(\mathcal{X},{\bm{\hat{\theta}}})],
\end{equation}
we find that  $\nabla_\mathcal{X} J$ is Lipschitz continuous if $\nabla_\mathcal{X}F(\mathcal{X},{\bm{\hat{\theta}}})$ is bounded and Lipschitz continuous. Therefore, it suffices to derive the regularity of $\nabla_\mathcal{X}F(\mathcal{X},{\bm{\hat{\theta}}})$. 

Using the definition of $F(\mathcal{X},{\bm{\hat{\theta}}})$ in \re{F1}, we have
\begin{equation*}
\begin{split}
    \nabla_\mathcal{X} F(\mathcal{X},{\bm{\hat{\theta}}}) =&\; \frac2m \sum_{i=1}^m \big(\mathcal{L}_\tau[\rho](\hat{\theta}_i)\big) \nabla_\mathcal{X}\mathcal{L}_\tau[\rho](\hat{\theta}_i) \\
    &\hspace{2em}+ 2\sum_{\alpha=0}^2 \big(D^\alpha \rho(0;\mathcal{X})-D^\alpha \rho(2\pi;\mathcal{X})\big)\cdot \nabla_\mathcal{X}\big(D^\alpha \rho(0;\mathcal{X})-D^\alpha\rho(2\pi;\mathcal{X})\big),
    \end{split}
\end{equation*}
where $\mathcal{L}_\tau [\rho](\hat{\theta}_i)$ is defined in \re{Ltauf}. For brevity, we denote the integrand in $\mathcal{L}_\tau [\rho](\hat{\theta}_i)$ to be function $M$, namely,
\begin{equation}\label{m}
    \begin{split}
       M(\theta,\hat{\theta}_i) =& \; M(\rho(\theta), \rho(\hat{\theta}_i),\rho'(\theta), \rho'(\hat{\theta}_i),\rho''(\theta), \rho''(\hat{\theta}_i))\\
       =&\; \beta G_1(D_\tau[\rho]) \sqrt{[\rho'(\theta)]^2 +  \rho^2(\theta)} - \Big((\mu+\kappa_\rho(\theta))Q(D_\tau[\rho]) - \frac{\kappa_\rho(\theta)-\kappa_\rho(\hat{\theta}_i)}{2\pi (D_\tau[\rho])^2}\Big)\\
   & \hspace{15em}\cdot\Big(\rho^2(\theta)+\rho(\hat{\theta}_i)\rho'(\theta)\sin(\hat{\theta}_i-\theta) - \rho(\hat{\theta}_i)\rho(\theta)\cos(\hat{\theta}_i-\theta)\Big).
    \end{split}
\end{equation}
We notice that the function $M$ involves $\rho$, $\rho'$, as well as $\rho''$. In our settings, 
\begin{equation}\label{f}
    \rho(\theta) = \sum a_i \Psi(b_i \theta + c_i) + d,
\end{equation}
hence
\begin{gather}
    \rho'(\theta) = \sum a_i b_i \Psi'(b_i \theta + c_i),\label{f'}\\
    \label{f''} \rho''(\theta) = \sum a_i b_i^2 \Psi''(b_i \theta + c_i).
\end{gather}
Clearly, for each $\hat{\theta}_i$,
\begin{equation}\label{gL}
    \nabla_\mathcal{X} \mathcal{L}_\tau[\rho](\hat{\theta}_i) = \int_0^{2\pi} \nabla_\mathcal{X} M(\theta,\hat{\theta}_i)\, \dif \theta,
\end{equation}
and, by the chain rules,
\begin{equation}
    \label{md}
\begin{split}
\nabla_\mathcal{X} M\; =& \;\frac{\p M}{\p \rho(\theta)}\nabla_\mathcal{X} \rho(\theta) + \frac{\p M}{\p \rho'(\theta)}\nabla_\mathcal{X}\rho'(\theta) + \frac{\p M}{\p \rho''(\theta)}\nabla_\mathcal{X}\rho''(\theta)\\
&+\frac{\p M}{\p \rho(\hat{\theta}_i)}\nabla_\mathcal{X} \rho(\hat{\theta}_i) + \frac{\p M}{\p \rho'(\hat{\theta}_i)}\nabla_\mathcal{X}\rho'(\hat{\theta}_i) + \frac{\p M}{\p \rho''(\hat{\theta}_i)}\nabla_\mathcal{X}\rho''(\hat{\theta}_i)
\end{split}
\end{equation}
Let's then deal with each term in \re{md}. It follows from \re{m} that
\begin{equation*}
    \begin{split}
\frac{\p M}{\p \rho(\theta)} = & \;  \beta G_1'(D_\tau[\rho])\frac{\p D_\tau[\rho]}{\p \rho(\theta)} \sqrt{[\rho'(\theta)]^2 +  \rho^2(\theta)}
        + \beta G_1(D_\tau[\rho])\frac{  \rho(\theta)}{\sqrt{[\rho'(\theta)]^2 +  \rho^2(\theta)}}- \Big(\big(\mu+\kappa_\rho(\theta)\big)  \\
        &Q'(D_\tau[\rho])+ \frac{\kappa_\rho(\theta)-\kappa_\rho(\hat{\theta}_i)}{\pi (D_\tau[\rho])^3}\Big) \frac{\p D_\tau[\rho]}{\p \rho(\theta)}\Big(\rho^2(\theta)+ \rho(\hat{\theta}_i)\rho'(\theta)\sin(\hat{\theta}_i-\theta)-\rho(\hat{\theta}_i)\rho(\theta)\cos(\hat{\theta}_i-\theta)\Big)\\
        &- \Big(Q(D_\tau[\rho])- \frac{1}{2\pi(D_\tau[\rho])^2}\Big)\frac{\p \kappa_\rho(\theta)}{\p \rho(\theta)}\Big(\rho^2(\theta)+ \rho(\hat{\theta}_i)\rho'(\theta)\sin(\hat{\theta}_i-\theta)-\rho(\hat{\theta}_i)\rho(\theta)\cos(\hat{\theta}_i-\theta)\Big)\\
        & - \Big(Q(D_\tau[\rho])\big(\mu+\kappa_\rho(\theta)\big)- \frac{\kappa_\rho(\theta)-\kappa_\rho(\hat{\theta}_i)}{2\pi (D_\tau[\rho])^2}\Big)\Big(2\rho(\theta)-\rho(\hat{\theta}_i)\cos(\hat{\theta}_i -\theta)\Big),\\
\frac{\p M}{\p \rho(\hat{\theta}_i)} = & \; \beta G_1'(D_\tau[\rho])\frac{\p D_\tau[\rho]}{\p \rho(\hat{\theta}_i)} \sqrt{[\rho'(\theta)]^2 +  \rho^2(\theta)} - \Big(\big(\mu+\kappa_\rho(\theta)\big)Q'(D_\tau[\rho]) + \frac{\kappa_\rho(\theta)-\kappa_\rho(\hat{\theta}_i)}{\pi (D_\tau[\rho])^3}\Big) \frac{\p D_\tau[\rho]}{\p \rho(\hat{\theta}_i)} \\
        &\Big(\rho^2(\theta) + \rho(\hat{\theta}_i)\rho'(\theta)\sin(\hat{\theta}_i-\theta)-\rho(\hat{\theta}_i)\rho(\theta)\cos(\hat{\theta}_i-\theta)\Big)- \frac{1}{2\pi(D_\tau[\rho])^2}\frac{\p \kappa_\rho(\hat{\theta}_i)}{\p \rho(\hat{\theta}_i)}\Big(\rho^2(\theta) \\
        &+\rho(\hat{\theta}_i)\rho'(\theta)\sin(\hat{\theta}_i-\theta)-\rho(\hat{\theta}_i)\rho(\theta)\cos(\hat{\theta}_i-\theta)\Big) - \Big(\big(\mu+\kappa_\rho(\theta)\big)Q(D_\tau[\rho])- \frac{\kappa_\rho(\theta)-\kappa_\rho(\hat{\theta}_i)}{2\pi (D_\tau[\rho])^2}\Big)\\
        &\Big(\rho'(\theta)\sin(\hat{\theta}_i - \theta)- \rho(\theta)\cos(\hat{\theta}_i-\theta)\Big),\\
\frac{\p M}{\p \rho'(\theta)} =&\; \beta G_1(D_\tau[\rho])\frac{\rho'(\theta)}{\sqrt{[\rho'(\theta)]^2 +  \rho^2(\theta)}} - \frac{\p \kappa_\rho(\theta)}{\p \rho'(\theta)}\Big( Q(D_\tau[\rho])- \frac{1}{2\pi (D_\tau[\rho])^2}\Big) \Big(\rho(\hat{\theta}_i)\rho'(\theta)\sin(\hat{\theta}_i-\theta)\\
    &+\rho^2(\theta)-\rho(\hat{\theta}_i)\rho(\theta)\cos(\hat{\theta}_i-\theta)\Big) - \Big(\big(\mu+\kappa_\rho(\theta)\big)Q(D_\tau[\rho])- \frac{\kappa_\rho(\theta)-\kappa_\rho(\hat{\theta}_i)}{2\pi (D_\tau[\rho])^2}\Big)\rho(\hat{\theta}_i)\sin(\hat{\theta}_i-\theta),\\
\frac{\p M}{\p \rho'(\hat{\theta}_i)} =&\;  - \frac{1}{2\pi (D_\tau[\rho])^2}\frac{\p \kappa_\rho(\hat{\theta}_i)}{\p \rho'(\hat{\theta}_i)} \Big(\rho^2(\theta)+ \rho(\hat{\theta}_i)\rho'(\theta)\sin(\hat{\theta}_i-\theta)-\rho(\hat{\theta}_i)\rho(\theta)\cos(\hat{\theta}_i-\theta)\Big),\\
\frac{\p M}{\p \rho''(\theta)} =&\;  - \frac{\p \kappa_\rho(\theta)}{\p \rho''(\theta)}\Big( Q(D_\tau[\rho])- \frac{1}{2\pi (D_\tau[\rho])^2}\Big) \Big(\rho^2(\theta)+ \rho(\hat{\theta}_i)\rho'(\theta)\sin(\hat{\theta}_i-\theta)-\rho(\hat{\theta}_i)\rho(\theta)\cos(\hat{\theta}_i-\theta)\Big),\\
\frac{\p M}{\p \rho''(\hat{\theta}_i)} =&\;  - \frac{1}{2\pi (D_\tau[\rho])^2}\frac{\p \kappa_\rho(\hat{\theta}_i)}{\p \rho''(\hat{\theta}_i)} \Big(\rho^2(\theta)+ \rho(\hat{\theta}_i)\rho'(\theta)\sin(\hat{\theta}_i-\theta)-\rho(\hat{\theta}_i)\rho(\theta)\cos(\hat{\theta}_i-\theta)\Big),
    \end{split}
\end{equation*}
where by \re{Dtauf}
\begin{eqnarray*}
    \frac{\p D_\tau[\rho]}{\p \rho(\theta)} =  \frac{\rho(\theta)-\rho(\hat{\theta}_i)\cos(\hat{\theta}_i-\theta)}{D_\tau[\rho]}, \hspace{2em} 
    \frac{\p D_\tau[\rho]}{\p \rho(\hat{\theta}_i)} =  \frac{\rho(\hat{\theta}_i) - \rho(\theta)\cos(\hat{\theta}_i-\theta)}{D_\tau[\rho]},
\end{eqnarray*}
and by \re{kappa12}
\begin{gather*}
    \frac{\p \kappa_\rho}{\p \rho} = \frac{-\rho^3 - 4\rho (\rho')^2 + 2\rho^2 \rho'' - \rho''(\rho')^2}{[\rho^2 +(\rho')^2]^{\frac52}},\\
    \frac{\p \kappa_\rho}{\p \rho'} = \frac{\rho'[\rho^2 - 2(\rho')^2 + 3\rho \rho'']}{[\rho^2 +(\rho')^2]^{\frac52}}, \hspace{2em} \text{and }\;\; \frac{\p \kappa_\rho}{\p \rho''} = \frac{-\rho}{[\rho^2 +(\rho')^2]^{\frac32}}.
\end{gather*}
Notice that 
\begin{equation*}
    D_\tau[\rho] \ge \tau > 0,
\end{equation*}
the singular point $r=0$ for $G_1(r)$, $G_1'(r)$, $Q(r)$, $Q'(r)$, $1/r^2$, and $1/r^3$ are not present in the above expression. Hence all the first-order derivatives of $M$ are bounded. Similarly, we can take another derivative to prove that all the second-order derivatives are also bounded. 
On the other hand,  $\nabla_\mathcal{X} \rho(\theta)$, $\nabla_\mathcal{X} \rho'(\theta)$, and $\nabla_\mathcal{X} \rho''(\theta)$ in \re{md} are computed by
\begin{gather*}
    \nabla_\mathcal{X} \rho(\theta) =  \Big(\cdots,\frac{\p \rho(\theta)}{\p a_i},\cdots,\frac{\p \rho(\theta)}{\p b_i},\cdots,\frac{\p \rho(\theta)}{\p c_i},\cdots, \frac{\p \rho(\theta)}{\p d}\Big),\\
    \nabla_\mathcal{X} \rho'(\theta) =  \Big(\cdots,\frac{\p \rho'(\theta)}{\p a_i},\cdots,\frac{\p \rho'(\theta)}{\p b_i},\cdots,\frac{\p \rho'(\theta)}{\p c_i},\cdots, \frac{\p \rho'(\theta)}{\p d}\Big),\\
    \nabla_\mathcal{X} \rho''(\theta) =  \Big(\cdots,\frac{\p \rho''(\theta)}{\p a_i},\cdots,\frac{\p \rho''(\theta)}{\p b_i},\cdots,\frac{\p \rho''(\theta)}{\p c_i},\cdots, \frac{\p \rho''(\theta)}{\p d}\Big).
\end{gather*}
From \re{f}, \re{f'}, as well as \re{f''}, we deduce 
\begin{align*}
        \frac{\p \rho(\theta)}{\p a_i} &= \Psi(b_i\theta + c_i), \hspace{2em} &\frac{\p \rho(\theta)}{\p b_i} &= a_i \theta \Psi'(b_i\theta+c_i),\\
        \frac{\p \rho(\theta)}{\p c_i} &= a_i\Psi'(b_i\theta + c_i),\hspace{2em} &\frac{\p \rho(\theta)}{\p d} &= 1;\\
        \frac{\p \rho'(\theta)}{\p a_i} &= b_i\Psi'(b_i\theta + c_i),\hspace{2em} &\frac{\p \rho'(\theta)}{\p b_i} &= a_i  \Psi'(b_i\theta+c_i)+ a_ib_i\theta \Psi''(b_i\theta+c_i),\\
        \frac{\p \rho'(\theta)}{\p c_i} &= a_ib_i\Psi''(b_i\theta + c_i), \hspace{2em}&\frac{\p \rho'(\theta)}{\p d} &= 0;\\
        \frac{\p \rho''(\theta)}{\p a_i} &= b_i^2 \Psi''(b_i\theta + c_i),\hspace{2em} &\frac{\p \rho''(\theta)}{\p b_i } &= 2a_i b_i \Psi''(b_i\theta + c_i) + a_i b_i^2 \theta \Psi'''(b_i\theta + c_i),\\
        \frac{\p \rho''(\theta)}{\p c_i} &= a_i b_i^2 \Psi'''(b_i\theta + c_i), \hspace{2em} &\frac{\p \rho''(\theta)}{\p d} &= 0.
\end{align*}
Similar calculations also work for $\nabla_\mathcal{X}\rho(\hat{\theta}_i)$, $\nabla_\mathcal{X}\rho'(\hat{\theta}_i)$, and $\nabla_\mathcal{X}\rho''(\hat{\theta}_i)$. 
Since the activation function $\Psi\in C^4$ in our assumptions, $\nabla_\mathcal{X} \rho$, $\nabla_\mathcal{X} \rho'$, and $\nabla_\mathcal{X} \rho''$ are all bounded. In the same manner, we can calculate second derivatives of $\rho$, $\rho'$, and $\rho''$, namely, the Hessian functions $\nabla_\mathcal{X}^2 \rho, \nabla_\mathcal{X}^2 \rho', \nabla_\mathcal{X}^2 \rho'': \, \mathbb{R}^d \rightarrow \mathbb{R}^{d\times d}$. We shall find that $\nabla_\mathcal{X}^2 \rho$, $\nabla_\mathcal{X}^2 \rho'$, $\nabla_\mathcal{X}^2 \rho''$ contain at most fourth-order derivatives of $\Psi$, hence they are all bounded as $\Psi\in C^4$. 

Combining the above analysis, we find that $\nabla_\mathcal{X} F(\mathcal{X},{\bm{\hat{\theta}}})$ is bounded. To further claim that it is Lipschitz continuous, we take another derivative with respect to $\mathcal{X}$ to derive
\begin{equation*}
\small
\begin{split}
    \nabla_\mathcal{X}^2 F(\mathcal{X},{\bm{\hat{\theta}}}) =&\; \frac2m \sum_{i=1}^m \Big[ \big(\nabla_\mathcal{X} \mathcal{L}_\tau [\rho](\hat{\theta}_i)\big)^T \nabla_\mathcal{X} \mathcal{L}_\tau [\rho](\hat{\theta}_i) + \big( \mathcal{L}_\tau [\rho](\hat{\theta}_i)\big)\nabla_\mathcal{X}^2  \mathcal{L}_\tau [\rho](\hat{\theta}_i)\Big]\\
    & \hspace{1em}+2\sum_{\alpha=0}^2 \Big(\nabla_\mathcal{X}\big(D^\alpha (\rho(0;\mathcal{X})-\rho(2\pi;\mathcal{X})\big)\Big)^T\Big(\nabla_\mathcal{X}\big(D^\alpha (\rho(0;\mathcal{X})-\rho(2\pi;\mathcal{X})\big)\Big)\\
    &\hspace{1em}+2\sum_{\alpha=0}^2 \Big(D^\alpha (\rho(0;\mathcal{X})-\rho(2\pi;\mathcal{X})\Big)\Big(\nabla_\mathcal{X}^2\big(D^\alpha (\rho(0;\mathcal{X})-\rho(2\pi;\mathcal{X})\big)\Big),\\
    \end{split}
\end{equation*}
where by \re{gL} and \re{md},
\begin{equation*}
\nabla_\mathcal{X}^2 \mathcal{L}_\tau[\rho](\hat{\theta}_i) = \int_0^{2\pi} \nabla_\mathcal{X}^2 M(\theta,\hat{\theta}_i) \,\dif \theta,
\end{equation*}
\begin{equation*}
    \nabla_\mathcal{X}^2 M = \frac{\p^2 M}{\p \rho(\theta)^2}\big(\nabla_\mathcal{X}\rho(\theta)\big)^T \nabla_\mathcal{X} \rho(\theta) + \frac{\p M}{\p \rho(\theta)}\nabla_\mathcal{X}^2 \rho(\theta) + 2\frac{\p^2 M}{\p \rho(\theta)\rho(\hat{\theta}_i)}\big(\nabla_\mathcal{X}\rho(\theta)\big)^T \nabla_\mathcal{X} \rho(\hat{\theta}_i) + \cdots.
\end{equation*}
According to the above analysis, each term in the above formula for $\nabla^2_\mathcal{X} F(\mathcal{X},{\bm{\hat{\theta}}})$ is bounded, hence $\nabla_\mathcal{X}F(\mathcal{X},{\bm{\hat{\theta}}})$ is Lipschitz continuous in $\mathcal{X}$. Using \re{FJ}, we conclude that the objective gradient function $\nabla_\mathcal{X}J(\mathcal{X})$ is also Lipschitz continuous.

\bigskip
\subsection{First and second moment limits}
Next, we shall prove the first and second moment limits condition in \cite{optimization}. The proof can be justified in three parts:

\vspace{5pt}

\noindent (a) According to our assumptions in Theorem \ref{main2}, $\{\mathcal{X}_k\}$ is contained in an open set which is bounded. From 4.1, $J(\mathcal{X})$ is continuously differentiable, hence $J(\mathcal{X}_k)$ is clearly bounded.

\vspace{5pt}

\noindent (b) Since $G(\mathcal{X}_k,{\bm{\hat{\theta}}}_k) =  \nabla_\mathcal{X} F(\mathcal{X}_k, {\bm{\hat{\theta}}}_k)$, we have
\begin{equation*}
    \mathbb{E}_{{\bm{\hat{\theta}}}_k}[G(\mathcal{X}_k,{\bm{\hat{\theta}}}_k)] =\mathbb{E}_{{\bm{\hat{\theta}}}_ k}[\nabla_{\mathcal{X}}F(\mathcal{X}_k,{\bm{\hat{\theta}}}_k)] = \nabla_\mathcal{X} \Big(\mathbb{E}_{{\bm{\hat{\theta}}}_k}[F(\mathcal{X}_k, {\bm{\hat{\theta}}}_k)]\Big) = \nabla_{\mathcal{X}} J(\mathcal{X}_k).
\end{equation*}
Therefore,
\begin{equation*}
    \nabla_\mathcal{X} J(\mathcal{X}_k)^T \mathbb{E}_{{\bm{\hat{\theta}}}_k}[G(\mathcal{X}_k,{\bm{\hat{\theta}}}_k)] = \nabla_\mathcal{X} J(\mathcal{X}_k)^T \cdot \nabla_\mathcal{X} J(\mathcal{X}_k) = \|\nabla_\mathcal{X} J(\mathcal{X}_k)\|_2^2.
\end{equation*}
It then directly follows that
\begin{equation*}
    \nabla_\mathcal{X} J(\mathcal{X}_k)^T \mathbb{E}_{{\bm{\hat{\theta}}}_k}[G(\mathcal{X}_k,{\bm{\hat{\theta}}}_k)] \ge u\|\nabla_\mathcal{X} J(\mathcal{X}_k)\|_2^2,
\end{equation*}
\begin{equation*}
    \text{and }\quad \|\mathbb{E}_{{\bm{\hat{\theta}}}_k}[G(\mathcal{X}_k,{{\bm{\hat{\theta}}}_k})]\|_2 = \|\nabla_\mathcal{X} J(\mathcal{X}_k)\|_2 \le u_G\|\nabla_\mathcal{X} J(\mathcal{X}_k)\|_2
\end{equation*}
hold true for some $0<u\le 1$ and $u_G\ge 1$. 

\vspace{5pt}

\noindent (c) Based on the analysis in 4.1, we know that $G(\mathcal{X}_k, {\bm{\hat{\theta}}}_k)= \nabla_\mathcal{X}F(\mathcal{X}_k,{\bm{\hat{\theta}}}_k)$ is bounded for each $k$, hence $\mathbb{E}_{{\bm{\hat{\theta}}}_k}[\|G(\mathcal{X}_k, {\bm{\hat{\theta}}}_k)\|_2^2]$ is also bounded. 
Since 
\begin{equation*}
    \mathbb{V}_{{\bm{\hat{\theta}}}_k}[G(\mathcal{X}_k,{\bm{\hat{\theta}}}_k)] = \mathbb{E}_{{\bm{\hat{\theta}}}_k}[\|G(\mathcal{X}_k, {\bm{\hat{\theta}}}_k)\|_2^2] - \|\mathbb{E}_{{\bm{\hat{\theta}}}_k} [G(\mathcal{X}_k,{\bm{\hat{\theta}}}_k)]\|^2_2 \le \mathbb{E}_{{\bm{\hat{\theta}}}_k}[\|G(\mathcal{X}_k, {\bm{\hat{\theta}}}_k)\|_2^2],
\end{equation*}
it indicates that $\mathbb{V}_{{\bm{\hat{\theta}}}_k}[G(\mathcal{X}_k,{\bm{\hat{\theta}}}_k)]$ is bounded.

\vspace{10pt}
We have verified the two sufficient assumptions in \cite{optimization}. 
Using the Theorem 4.10 in \cite{optimization}, with the diminishing step-size, i.e.,
\begin{equation*}
    \sum\limits_{k=1}^\infty \alpha_k = \infty \;\text{  and  }\; \sum\limits_{k=1}^\infty \alpha_k^2 < \infty,
\end{equation*}
we get the convergence result, i.e., $\mathbb{E}\Big[\frac{1}{A_k} \sum_{k=1}^K \alpha_k \|\nabla F(\mathcal{X})\|_2^2\Big] \rightarrow 0$ as $K\rightarrow \infty$.
\end{proof}

The stochastic gradient decent method generates critical points.
We are only interested in those critical points whose loss
function is close to zero, hence generating an approximate solution to our free boundary problem -- this is achieved in 
our numerical examples.

\section{Numerical Results}

\subsection{Verification of the neural network discretization near bifurcation points} 
Near the bifurcation points $\mu = \mu_n(R_s)$, the shape 
of the symmetry-breaking free boundary is fully characterized by
Theorem \ref{bifur} (see also Remark \ref{rem1}). In this section we show that all these 
free boundary solutions can be fully recovered  by the neural network discretization. In particular, 
we compute the numerical solution with  Algorithm \ref{alg} near the bifurcation points $\mu_n$  by \re{bifurps} with $R_S = 1$, namely, 
$$\mu_2 \approx 14.7496, \hspace{2em} \mu_3 \approx 28.7234, \hspace{2em} \mu_4 \approx 47.1794, \hspace{2em} \mu_5 \approx 70.1169.$$
We choose $\mu$ in a small neighborhood of $\mu_n$, i.e., $|\mu - \mu_n|$ is small; correspondingly, $\beta = (\mu+1)\frac{I_1(1)}{I_0(1)}$ is uniquely determined by \re{beta}. For the neural network discretization, we set the number of neurons  $N=20$, 
the number of Monte Carlo integration points $m=4000$, $\tau = 10^{-3}$, the maximum number of iterations as 50, learning rate = $10^{-4}$, and the activation function $\Psi(\theta) = \cos(\theta)$. The initial parameters are set to be: $a_i$ is randomly chosen by $\mathbb{N}(0,1)$, $b_i = n$ (which corresponds to the $n$-mode bifurcation), $c_i = 0$, and $d=1$; in this way, the Neural network representation $\rho(\theta,\mathcal{X}) = \sum_{i=1}^N a_i \Psi(b_i\theta+c_i) + d$ is close to the form of free boundary for the symmetry-breaking solution \re{form}. Moreover, $\hat{\theta}_i$ are uniformly sampled from $[0,2\pi]$, and are divided into 20 mini-batches, with each mini-batch containing 200 points. Therefore, all the parameters are updated 20 times in one epoch.  The loss is shown in Figure \ref{lossplot}  while the shapes of symmetry-breaking solutions on different bifurcation branches are shown in  Figure \ref{bifurplot}  which is consistent with the theoretical results in Theorem \ref{bifur} and Remark \ref{rem1}.

\begin{figure}[ht]
\begin{subfigure}{.48\textwidth}
  \centering
  \includegraphics[width=.9\linewidth,trim={3cm 0.3cm 3cm 0.3cm}]{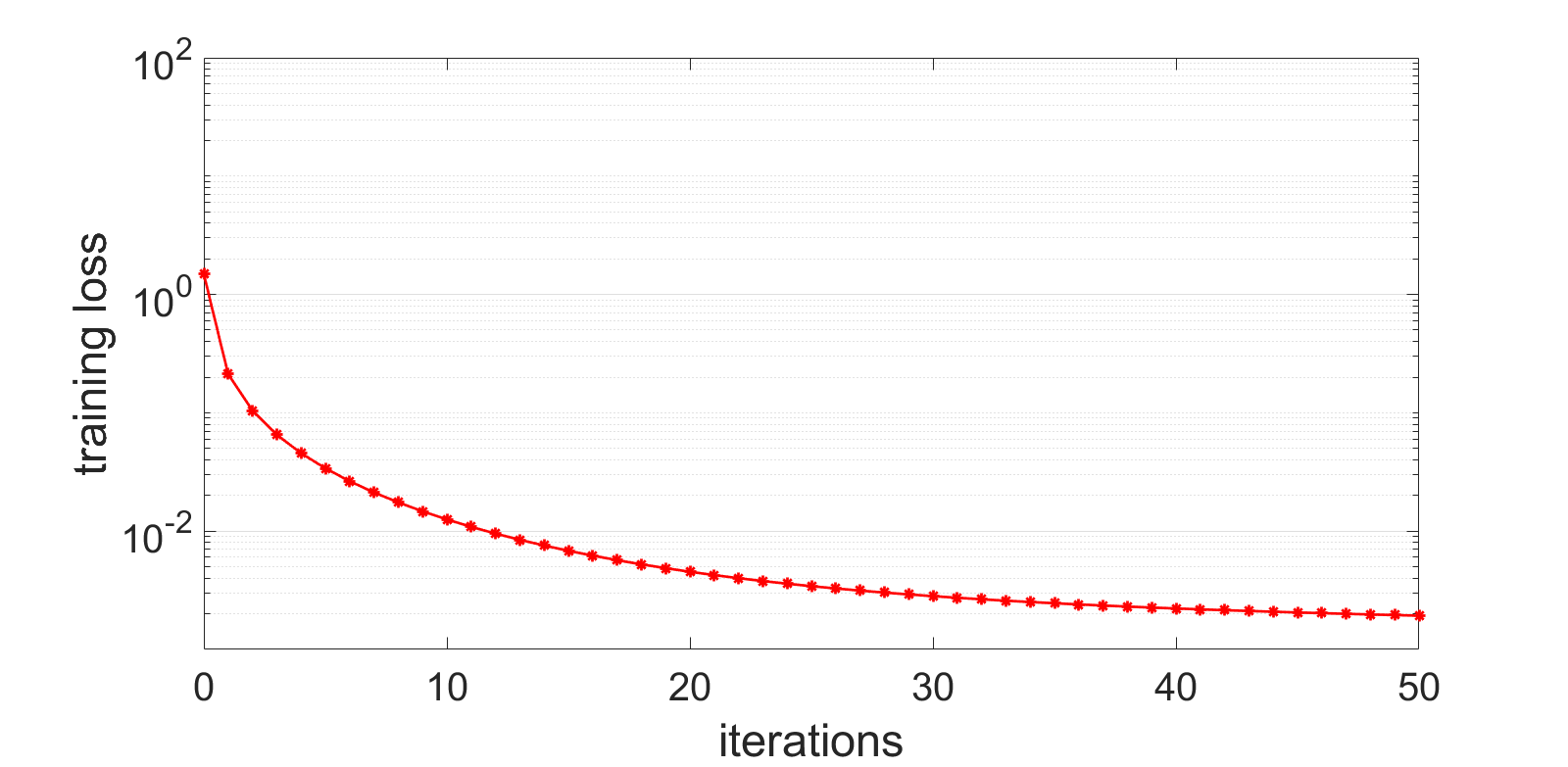}
  \caption{$n=2$ bifurcation, $\mu=14.6$.}
\end{subfigure}
\begin{subfigure}{.48\textwidth}
  \centering
  \includegraphics[width=.9\linewidth,trim={3cm 0.3cm 3cm 0.3cm}]{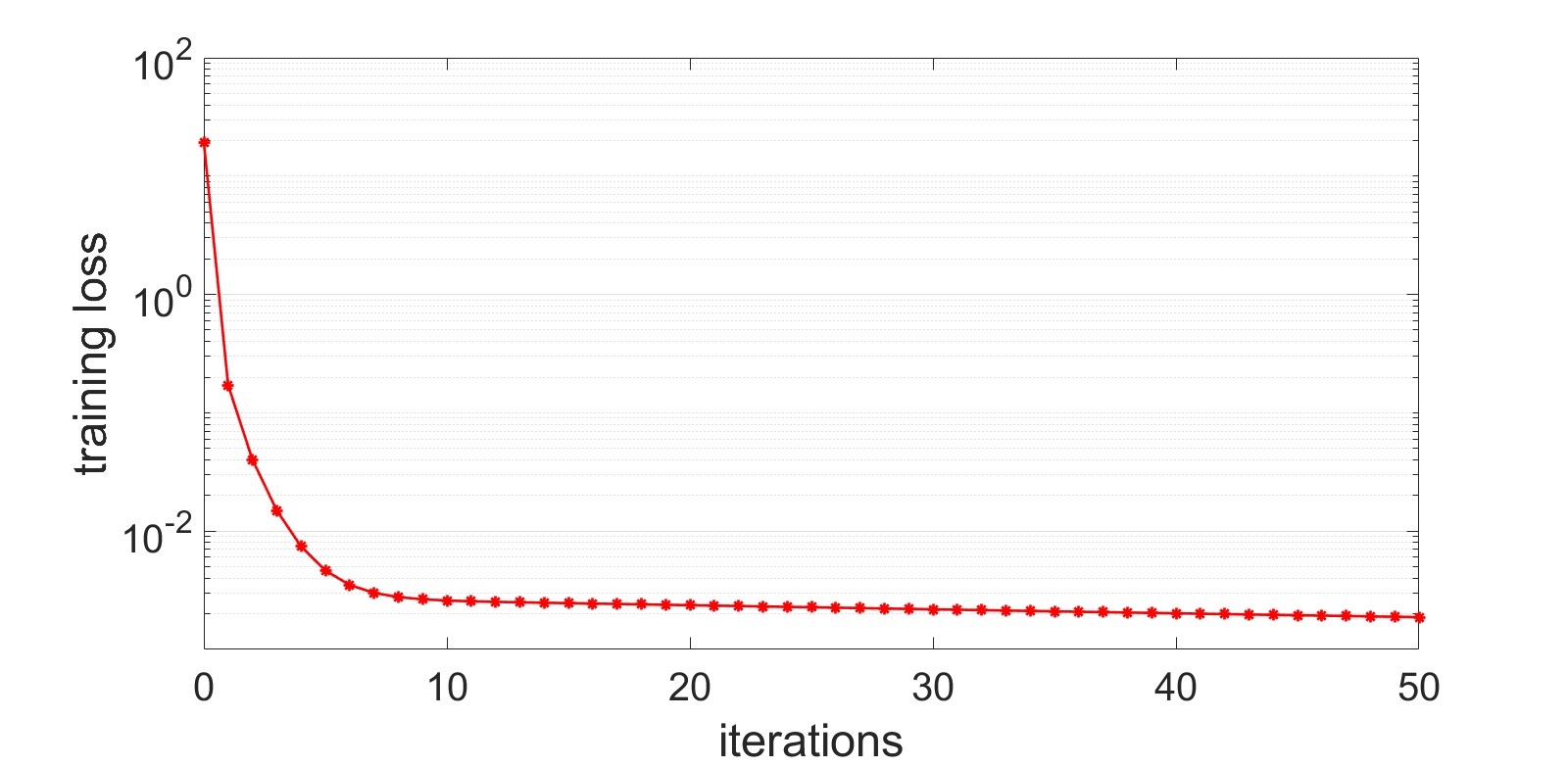}
  \caption{$n=3$ bifurcation, $\mu=28.6$.}
\end{subfigure}

\vspace{10pt}
\begin{subfigure}{.48\textwidth}
  \centering
  \includegraphics[width=.9\linewidth,trim={3cm 0.3cm 3cm 0.3cm}]{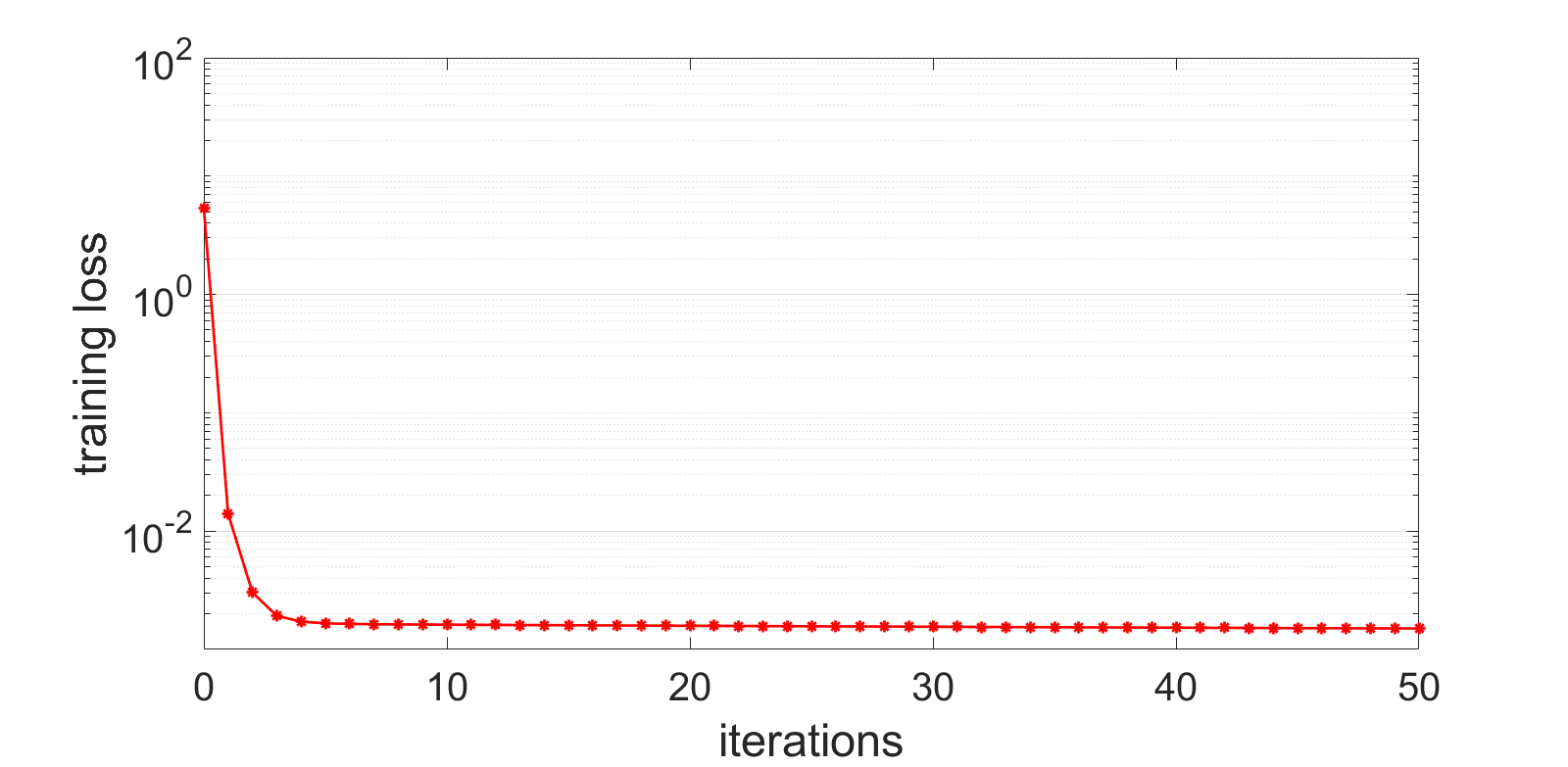}
  \caption{$n=4$ bifurcation, $\mu=47.0$.}
\end{subfigure}
\begin{subfigure}{.48\textwidth}
  \centering
  \includegraphics[width=.9\linewidth,trim={3cm 0.3cm 3cm 0.3cm}]{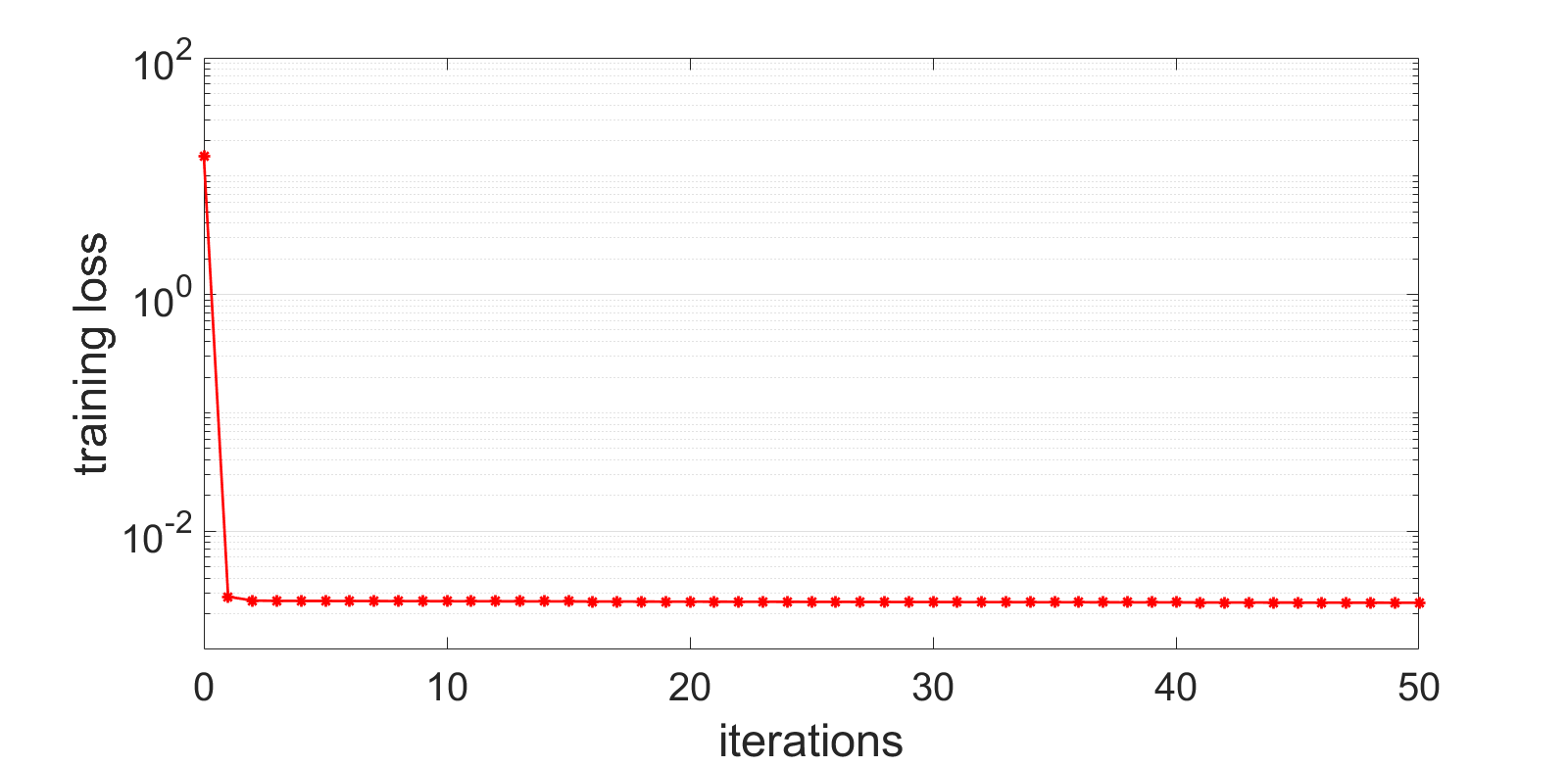}  
  \caption{$n=5$ bifurcation, $\mu=70.0$.}
\end{subfigure}
\caption{Training loss.}
\label{lossplot}
\end{figure}

\begin{figure}[ht]
\begin{subfigure}{.48\textwidth}
  \centering
  \includegraphics[width=.7\linewidth,trim={1cm 0.3cm 1cm 0.3cm}]{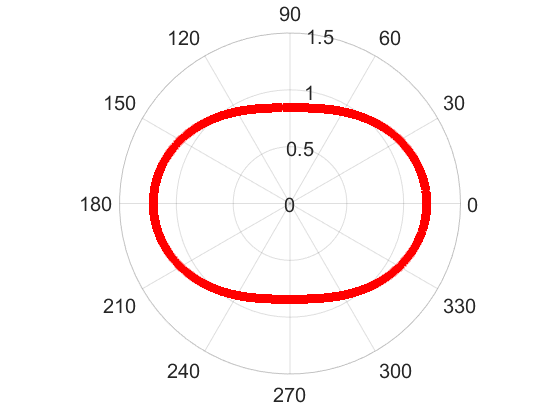}  
  \caption{$n=2$ bifurcation, $\mu=14.6$.}
\end{subfigure}
\begin{subfigure}{.48\textwidth}
  \centering
  \includegraphics[width=.7\linewidth,trim={1cm 0.3cm 1cm 0.3cm}]{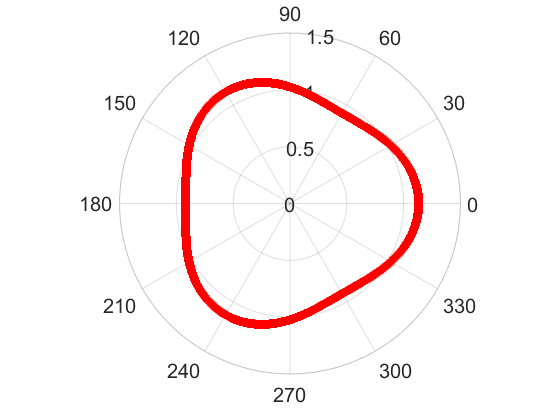}  
  \caption{$n=3$ bifurcation, $\mu=28.6$.}
\end{subfigure}

\vspace{10pt}
\begin{subfigure}{.48\textwidth}
  \centering
  \includegraphics[width=.7\linewidth,trim={1cm 0.3cm 1cm 0.3cm}]{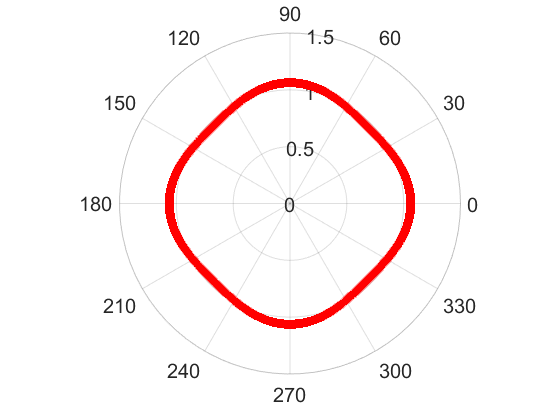}  
  \caption{$n=4$ bifurcation, $\mu=47.0$.}
\end{subfigure}
\begin{subfigure}{.48\textwidth}
  \centering
  \includegraphics[width=.7\linewidth,trim={1cm 0.3cm 1cm 0.3cm}]{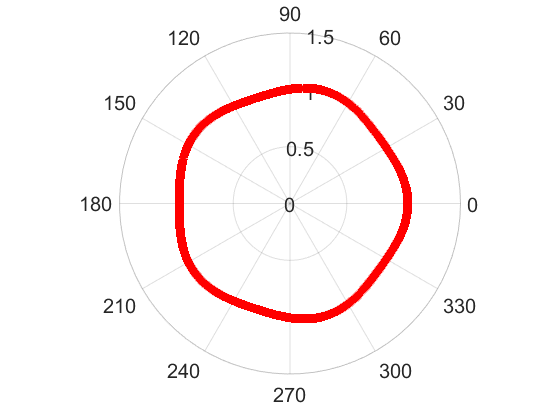}  
  \caption{$n=5$ bifurcation, $\mu=70.0$.}
\end{subfigure}
\caption{Contour plot of non-radially symmetric solutions in different bifurcation branches.}
\label{bifurplot}
\end{figure}

\subsection{Other non-radially symmetric solutions}
In this section we generate some non-radially symmetric solutions
that are not characterized by any theorems. Inspired by \cite{viscosity}, we try to find some fingering patterns, hence we choose the activation function $\Psi(\theta) = 0.3/[(\cos(\theta))^2 + (0.3\sin(\theta))^2]$, which generates fingering-like patterns. In particular, we take $\mu = 20$, $\beta = (\mu+1)\frac{I_1(1)}{I_0(1)}$, $N=20$, $m = 10000$, $\tau=10^{-3}$ and maximum number of iterations = 200. We divide 10000 random points $\hat{\theta}_i$ into 100 mini-batches, hence all the parameters are updated 100 times in one iteration. 

In Figure \ref{2f}, we initially choose  $b_i = 1$ and randomly choose  other parameters. The learning rate is $10^{-3}$ at first and is decreased gradually to $10^{-6}$. In Figure \ref{4f}, we take  $b_i = 2$, $c_i = 0$, random $a_i$, a random $d$, and a $10^{-5}$ learning rate. 
Compared with Figure \ref{lossplot}, the loss in Figures \ref{2f} and \ref{4f} are larger. It is due to the numerical error introduced by calculating the curvature at the tip of each finger and the connecting points between two adjacent fingers.

\begin{figure}[H]
\begin{subfigure}{.48\textwidth}
  \centering
  \includegraphics[height=1.5in,trim={2cm 0.3cm 2cm 0.3cm}]{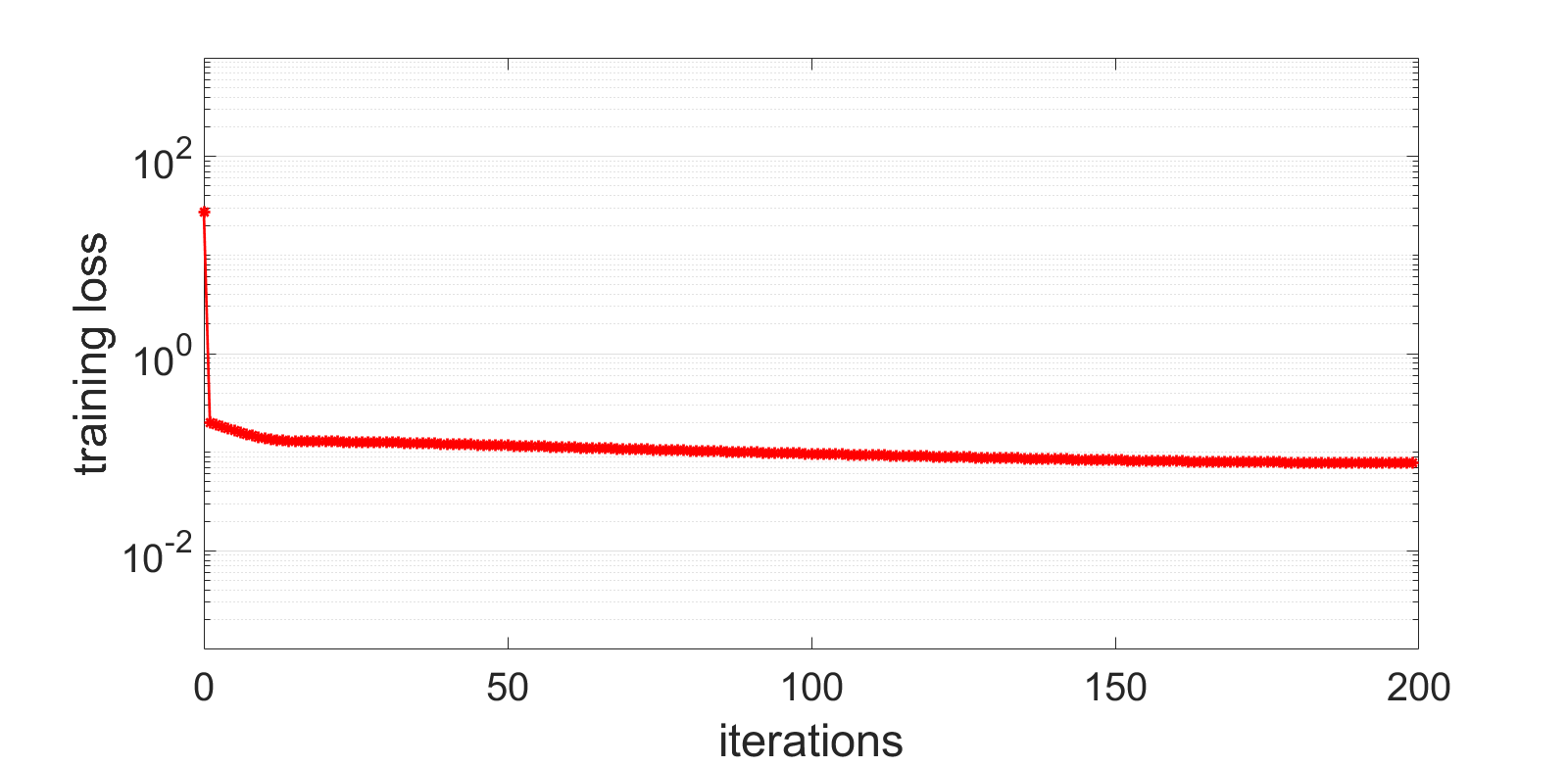}
  \caption{Training loss.}
\end{subfigure}
\begin{subfigure}{.48\textwidth}
  \centering
  \includegraphics[height=1.5in,trim={2cm 0.3cm 2cm 0.3cm}]{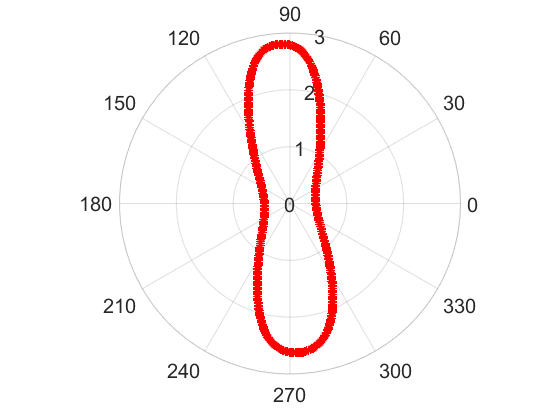}  
  \caption{Contour plot.}
\end{subfigure}
\caption{Non-radially symmetric solution with 2 fingers.}
\label{2f}
\end{figure}

\begin{figure}[ht]
\begin{subfigure}{.48\textwidth}
  \centering
  \includegraphics[height=1.5in,trim={2cm 0.3cm 2cm 0.3cm}]{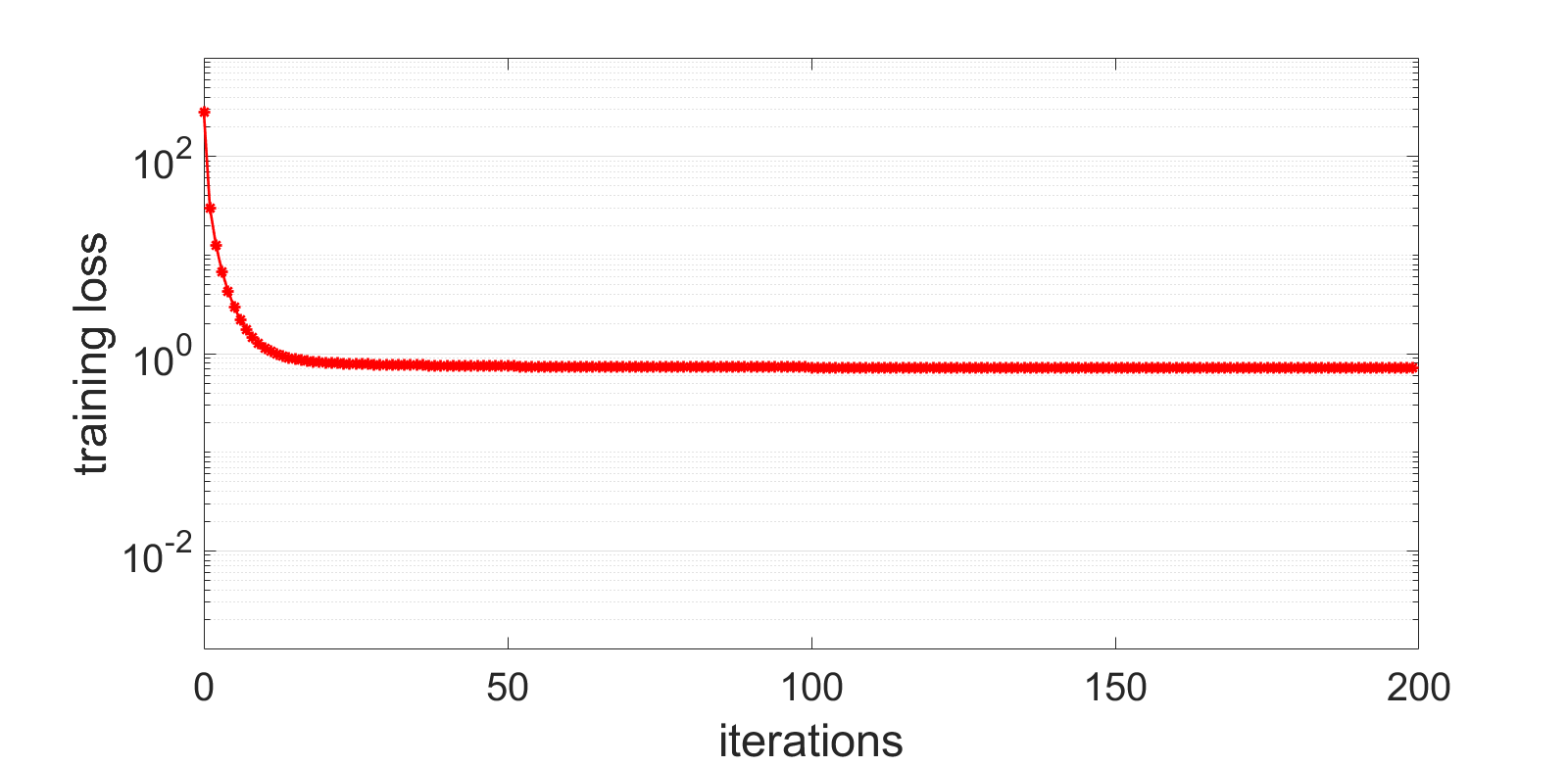}
  \caption{Training loss.}
\end{subfigure}
\begin{subfigure}{.48\textwidth}
  \centering
  \includegraphics[height=1.5in,trim={2cm 0.3cm 2cm 0.3cm}]{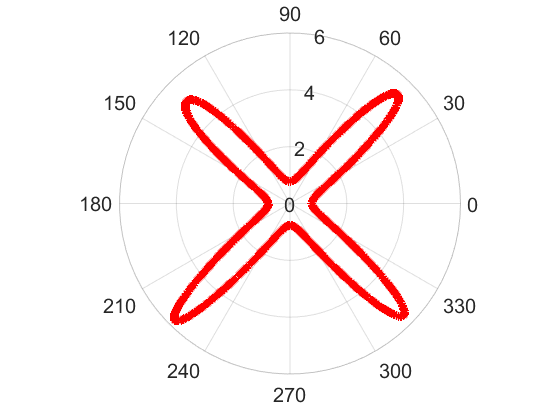}  
  \caption{Contour plot.}
\end{subfigure}
\caption{Non-radially symmetric solution with 4 fingers.}
\label{4f}
\end{figure}

\section{Conclusion}
We have developed a novel numerical method based on the neural network discretization for solving
a   modified Hele-Shaw model of PDE free boundary problem. 
We established theoretically the existence of the numerical
solution with this new discretization. Our simulations verify this new approach on
  radially symmetric and known 
non-radially symmetric solutions. Moreover, using this new method, we also found some new non-radially symmetric
solutions that were unknown based on the existing theories. 
In the future, we will apply this new numerical method to solve the more sophisticated free boundary problems such as tumor growth models or the atherosclerotic plaque formation models.

\section{Appendix}
\begin{thm}{\bf (Crandall-Rabinowitz theorem, \cite{crandall})}\label{bifurthm}
Let $X$, $Y$ be real Banach spaces and $F(\cdot,\cdot)$ a $C^p$ map, $p\ge 3$, of a neighborhood $(0,\mu_0)$ in $X \times \mathbb{R}$ into $Y$. Suppose
\begin{itemize}
\item[(1)] $F(0,\mu) = 0$ for all $\mu$ in a neighborhood of $\mu_0$,
\item[(2)] $\mathrm{Ker} \,F_x(0,\mu_0)$ is one dimensional space, spanned by $x_0$,
\item[(3)] $\mathrm{Im} \,F_x(0,\mu_0)=Y_1$ has codimension 1,
\item[(4)] $F_{\mu x}(0,\mu_0) x_0 \notin Y_1$.
\end{itemize}
Then $(0,\mu_0)$ is a bifurcation point of the equation $F(x,\mu)=0$ in the following sense: In a neighborhood of $(0,\mu_0)$ the set of solutions $F(x,\mu) =0$ consists of two $C^{p-2}$ smooth curves $\Gamma_1$ and $\Gamma_2$ which intersect only at the point $(0,\mu_0)$; $\Gamma_1$ is the curve $(0,\mu)$ and $\Gamma_2$ can be parameterized as follows:
$$\Gamma_2: (x(\varepsilon),\mu(\varepsilon)), |\varepsilon| \text{ small, } (x(0),\mu(0))=(0,\mu_0),\; x'(0)=x_0.$$
\end{thm}

\end{document}